\documentclass[twoside,a4paper,reqno,11pt]{amsart} 
\usepackage[top=30mm,right=30mm,bottom=30mm,left=30mm]{geometry}

\usepackage{amsfonts, amsmath, amssymb, mathrsfs, bm, latexsym, stmaryrd, array, hyperref, mathtools}

\numberwithin{equation}{section}

\renewcommand{\a}{\alpha}
\renewcommand{\b}{\beta}
\newcommand{\g}{\gamma}
\newcommand{\e}{\varepsilon}

\renewcommand{\O}{\Omega}

\newcommand{\F}{\mathbb{F}_{q}}
\newcommand{\Nat}{\mathbb{N}}

\newcommand{\M}{\mathcal{M}}

\newcommand{\<}{\langle}
\renewcommand{\>}{\rangle}
\newcommand{\la}{\langle}
\newcommand{\ra}{\rangle}

\renewcommand{\to}{\rightarrow}

\newcommand{\leqs}{\leqslant}
\newcommand{\geqs}{\geqslant}

\newcommand{\fpr}{\mbox{{\rm fpr}}}

\newcommand{\what}{\widehat} 

\makeatletter
\newcommand{\imod}[1]{\allowbreak\mkern4mu({\operator@font mod}\,\,#1)}
\makeatother

\newcommand{\vs}{\vspace{3mm}}

\newtheorem{theorem}{Theorem}

\newtheorem{thm}{Theorem}[section] 
\newtheorem{lem}[thm]{Lemma}
\newtheorem{prop}[thm]{Proposition}
 
\newtheorem{cor}[thm]{Corollary}

\theoremstyle{definition}
\newtheorem{remk}{Remark}
\newtheorem{rem}[thm]{Remark}

\newtheorem*{def-non}{Definition}

\begin{document}

\author{Timothy C. Burness}
\address{T.C. Burness, School of Mathematics, University of Bristol, Bristol BS8 1TW, UK}
\email{t.burness@bristol.ac.uk}
 
\author{Scott Harper}
\thanks{The second author thanks the Engineering and Physical Sciences Research Council and the Heilbronn Institute for Mathematical Research for their financial support. Both authors thank an anonymous referee for helpful comments on a previous version of the paper.}
\address{S. Harper, School of Mathematics, University of Bristol, Bristol BS8 1TW, UK}
\email{scott.harper@bristol.ac.uk}
 
\title[On the uniform domination number of a finite simple group]{On the uniform domination number \\ of a finite simple group}

\date{\today}

\begin{abstract}
Let $G$ be a finite simple group. By a theorem of Guralnick and Kantor, $G$ contains a conjugacy class $C$ such that for each non-identity element $x \in G$, there exists $y \in C$ with $G = \la x,y\ra$. Building on this deep result, we introduce a new invariant $\gamma_u(G)$, which we call the uniform domination number of $G$. This is the minimal size of a subset $S$ of conjugate elements such that for each $1 \ne x \in G$, there exists $s \in S$ with $G = \la x, s\ra$. (This invariant is closely related to the total domination number of the generating graph of $G$, which explains our choice of terminology.) By the result of Guralnick and Kantor, we have $\gamma_u(G) \leqs |C|$ for some conjugacy class $C$ of $G$, and the aim of this paper is to determine close to best possible bounds on $\gamma_u(G)$ for each family of simple groups. For example, we will prove that there are infinitely many non-abelian simple groups $G$ with $\gamma_u(G) = 2$. To do this, we develop a probabilistic approach, based on fixed point ratio estimates. We also establish a connection to the theory of bases for permutation groups, which allows us to apply recent results on base sizes for primitive actions of simple groups.
\end{abstract}

\maketitle

\setcounter{tocdepth}{1}
\tableofcontents


\section{Introduction}\label{s:intro}
The study of generators for simple groups has a long and rich history, with numerous applications. As a consequence of the Classification of Finite Simple Groups, it is known that every finite simple group can be generated by two elements; this is a theorem of Steinberg \cite{St} for groups of Lie type, and the argument was completed by Aschbacher and Guralnick in \cite{AG}.
This result leads to many interesting problems that have been the focus of intensive research in recent years. For instance, it is natural to consider the abundance of generating pairs in a simple group, and also the existence of generators with prescribed properties, such as restrictions on the orders of the generating elements. 

Through the work of many authors, we now understand that finite simple groups have some remarkable generation properties. For example, a theorem of Liebeck and Shalev \cite{LSh}, extending earlier work of Dixon \cite{Dix} and Kantor and Lubotzky \cite{KLub}, shows that a randomly chosen pair of elements in a finite simple group $G$ forms a generating set with probability tending to $1$ as $|G|$ tends to infinity. In \cite{GK} (also see \cite{Stein}), Guralnick and Kantor use probabilistic methods to prove that every non-identity element of a finite simple group $G$ belongs to a generating pair (a group with this strong $2$-generation property is said to be \emph{$\frac{3}{2}$-generated}). See \cite{BGK, BG, GSh, Harper} for further results in this direction for simple and almost simple groups. We refer the reader to \cite{B_sur} for a recent survey of related topics concerning the generation of simple groups.

Let $G$ be a finite group and let $G^{\#}$ be the set of non-identity elements of $G$. The \emph{generating graph} of $G$, denoted by $\Gamma(G)$, has vertex set $G^{\#}$ and two vertices are adjacent if and only if they generate $G$. This graph encodes many interesting generation properties of a $2$-generated group. For example, $G$ is $\frac{3}{2}$-generated if and only if $\Gamma(G)$ has no isolated vertices. In turn, many natural invariants of this graph have interesting group-theoretic interpretations, and this provides an appealing interplay between group theory and graph theory. For instance, it is natural to consider the connectedness, diameter and Hamiltonicity of $\Gamma(G)$, as well as its clique, co-clique and chromatic numbers. In recent years, numerous authors have focussed on these problems in the context of a non-abelian finite simple group $G$. Here one of the most striking results is \cite[Theorem 1.2]{BGK}, which implies that $\Gamma(G)$ is connected with diameter $2$. In \cite{BGLMN}, it is conjectured that $\Gamma(G)$ always contains a Hamiltonian cycle, but so far this has only been established for all sufficiently large simple groups (see \cite[Theorem 1.2]{BGLMN}). The proof of this result uses a combination of probabilistic and combinatorial techniques.

In this paper we initiate the study of another natural invariant of the generating graph of a finite group. Let $\Gamma$ be a finite graph with no isolated vertices. A subset $S$ of $\Gamma$ is a \emph{total dominating set} if every vertex of $\Gamma$ is adjacent to a vertex in $S$, and the \emph{total domination number} of $\Gamma$ is the minimal size of a total dominating set. This is a well-studied invariant, which, in general, is rather difficult to compute precisely. Indeed, the problem of determining whether the total domination number of a given graph is at most a given number $k$ is \textsf{NP}-complete (see the survey \cite{Hen} for more details). 

As noted above, if a finite group $G$ is $\frac{3}{2}$-generated then its generating graph $\Gamma(G)$ has no isolated vertices. In this situation, we define the \emph{total domination number} $\g_t(G)$ of $G$ to be the total domination number of $\Gamma(G)$. 

In this paper, we will work with a slightly stronger notion. Let $k$ be a positive integer. Following \cite{BGK}, we say that $G$ has \emph{uniform spread} $k$ if there exists a fixed conjugacy class $C$ of $G$ with the property that for any $k$ elements $x_1,\ldots,x_k \in G^{\#}$ there exists $g \in C$ such that  $G = \<x_i,g\>$ for all $i$. Therefore, $G$ has uniform spread $1$ if and only if some conjugacy class of $G$ is a total dominating set for $\Gamma(G)$. By the main theorem of \cite{GK}, every finite simple group $G$ has uniform spread $1$ (in fact, \cite[Theorem 1.2]{BGK} shows that all finite simple groups have uniform spread $2$). Therefore, for finite groups with uniform spread $1$, such as simple groups, it is natural to seek small total dominating sets of conjugate elements. This leads us naturally to the following definition.

\begin{def-non}\label{d:udn}
Let $G$ be a finite group with uniform spread $1$ and generating graph $\Gamma(G)$. We define the \emph{uniform domination number} $\gamma_u(G)$ of $G$ to be the minimal number of conjugate elements that form a total dominating set for $\Gamma(G)$. 
\end{def-non}

Observe that $\gamma_t(G) \leqs \gamma_u(G)$. Also note that $\g_u(G) = 1$ if and only if $G$ is cyclic.

We are now in a position to state our main results on the uniform domination number of simple groups. By the above observations, if $G$ is a non-abelian finite simple group then  
\begin{equation}\label{e:tr}
2 \leqs \gamma_u(G) \leqs |C|
\end{equation}
for some conjugacy class $C$ of $G$. (Typically, $C$ is large, such as a class of regular semisimple elements if $G$ is a group of Lie type.) Our first result shows that there are infinitely many groups for which the trivial lower bound in \eqref{e:tr} is sharp (see Theorems \ref{t:main_alt} and \ref{t:ExceptionalUDN}(i)).

\begin{theorem}\label{t:main1}
There are infinitely many non-abelian finite simple groups $G$ with $\gamma_u(G) = 2$. For example, $\gamma_u(A_n) = 2$ for every prime number $n \geqs 13$.
\end{theorem}

Next we present results for alternating, sporadic and groups of Lie type, in turn. Our main result for alternating groups is the following (see Theorem \ref{t:mainn} for a more detailed statement).

\begin{theorem}\label{t:main2}
There exists an absolute constant $c$ such that 
\[
\gamma_u(A_n) \leqs c (\log_2 n)
\]
for all $n \geqs 5$. In particular, if $n \geqs 6$ is even, then
\[
\lceil \log_2 n \rceil - 1 \leqs \gamma_u(A_n) \leqs 2\lceil \log_2 n \rceil.
\]
\end{theorem}

\begin{remk}
Notice that if $n$ is even then Theorem \ref{t:main2} gives the exact value of $\gamma_u(A_n)$, up to a small constant. It is also worth noting that the uniform domination number of an alternating group can be arbitrarily large. The analysis of odd degree alternating groups is more difficult and our best estimate is $\gamma_u(A_n) \leqs 77\log_2n$ (see Proposition \ref{t:altodd}).
\end{remk}

We can compute precise results for sporadic simple groups; a simplified version of our main result (Theorem \ref{t:SporadicUDN}) is as follows.

\begin{theorem}\label{t:main3}
Let $G$ be a sporadic simple group. Then $\gamma_u(G) \leqs 4$, with equality if $G = {\rm M}_{11}$ or ${\rm M}_{12}$. 
\end{theorem}

Finally, we present a version of our main result for simple groups of Lie type (see Theorems \ref{t:ExceptionalUDN} and \ref{t:ClassicalUDN} for more detailed results). In the statement, $r$ is the untwisted Lie rank of $G$ (that is, $r$ is the rank of the ambient simple algebraic group). 

\begin{theorem}\label{t:main4}
Let $G$ be a finite simple group of Lie type of rank $r$.
\begin{itemize}\addtolength{\itemsep}{0.2\baselineskip}
\item[{\rm (i)}] If $G = {\rm L}_{2}(q)$, then $\gamma_u(G) \leqs 4$, with equality if and only if $q=9$. 
\item[{\rm (ii)}] If $G$ is an exceptional group of Lie type, then $\gamma_u(G) \leqs 6$.
\item[{\rm (iii)}] If $G$ is a classical group, then 
$\gamma_u(G) \leqs 7r+56$.
\end{itemize}
\end{theorem}

\begin{remk}
Let us make some comments on the statement of Theorem~\ref{t:main4}.
\begin{itemize}\addtolength{\itemsep}{0.2\baselineskip}
\item[{\rm (a)}] The upper bound in part (ii) can be improved for some families of exceptional groups. For instance, by Theorem~\ref{t:ExceptionalUDN}, $\gamma_u(G) = 2$ if $G \in \left\{ {}^2B_2(q), {}^2G_2(q), E_8(q) \right\}$. 
\item[{\rm (b)}] We refer the reader to Theorem \ref{t:ClassicalUDN} for a more detailed version of part (iii), which provides stronger bounds in some special cases. For example, if $G$ is a symplectic group in even characteristic, or an odd dimensional orthogonal group, then 
\[
r \leqs \gamma_u(G) \leqs 7r
\]
and thus the linear bound in (iii) is essentially best possible (up to constants). In other cases, we can establish a constant bound. For instance, if $G = {\rm U}_{r+1}(q)$ and $r \geqs 7$ is odd, then $\gamma_u(G) \leqs 15$.
\end{itemize}
\end{remk}

Let us briefly describe some of the main ideas in the proofs of Theorems \ref{t:main1}--\ref{t:main4}. Following Guralnick and Kantor \cite{GK} in their work on the uniform spread of simple groups, we develop a probabilistic approach to study the uniform domination number. Let $G$ be a finite group and fix an element $s \in G^{\#}$. Write $\M(G,s)$ for the set of maximal subgroups of $G$ containing $s$. For an element $x \in G$ and subgroup $H < G$, let $\fpr(x,G/H)$ be the \emph{fixed point ratio} of $x$ for the action of $G$ on the set of cosets $G/H$. For a positive integer $c$, let $Q(G,s,c)$ be the probability that a randomly chosen $c$-tuple of conjugates of $s$ is \emph{not} a total dominating set for $\Gamma(G)$. Clearly, if $Q(G,s,c)<1$ for some $s \in G^{\#}$ then $\gamma_u(G) \leqs c$, so we are interested in bounding $Q(G,s,c)$ from above.

Here the key tool is Lemma~\ref{l:ProbMethod}, which states that 
\[
Q(G,s,c) \leqs \sum_{i=1}^{k} |x_i^G| \left( \sum_{H \in \M(G,s)} \fpr(x_i,G/H) \right)^c =: \what{Q}(G,s,c),
\]
where $\{x_1, \ldots, x_k\}$ is a set of representatives of the conjugacy classes in $G$ of prime order elements. To apply this result, the first step is to identify an element $s \in G^{\#}$ that is contained in very few maximal subgroups of $G$. We then need to determine the specific subgroups in $\M(G,s)$ and compute upper bounds for the relevant fixed point ratios. Here we can appeal to the extensive literature on fixed point ratios for simple groups. For example, if $G$ is a simple group of Lie type over $\F$ and $H$ is a maximal subgroup of $G$, then a well known theorem of Liebeck and Saxl \cite[Theorem 1]{LSax} implies that $\fpr(x,G/H) \leqs 4/3q$ for all $x \in G^{\#}$, with a short list of known exceptions. Stronger bounds are established in \cite{Bur} (for non-subspace actions of classical groups), \cite[Section 3]{GK} (subspace actions) and \cite{LLS} (exceptional groups).  

In the special case where there is an element $s \in G^{\#}$ with $\M(G,s) = \{H\}$ and $H$ is core-free, it is easy to see that $\gamma_u(G) \leqs b$, where $b = b(G,G/H)$ is the \emph{base size} of $G$ with respect to the action on $G/H$ (that is, $b$ is the minimal size of a subset of $G/H$ with trivial pointwise stabiliser). This observation provides an important connection between the uniform domination number of $G$ and the base sizes of primitive permutation representations of $G$. Bases for primitive groups have been a topic of interest in group theory since the nineteenth century, with a  wide range of applications. In particular, strong upper bounds on the base sizes of primitive almost simple groups have recently been established (see \cite{B07,BGS,BLS,BOW,Hal} for example), and in many cases we can apply these results to bound the uniform domination number. For example, Halasi's results \cite{Hal} on the base size for the action of a symmetric group on $k$-sets are a key ingredient in the proof of the bounds in Theorem~\ref{t:main2} for even degree alternating groups.

\begin{remk}
In many cases, our probabilistic approach also yields strong asymptotic results. Indeed, if $\what{Q}(G,s,c) \to 0$ as $|G| \to \infty$, then almost every $c$-tuple of conjugates of $s$ is a total dominating set for $\Gamma(G)$. For instance, suppose $G$ is an exceptional group of Lie type, in which case Theorem \ref{t:main4}(ii) gives $\gamma_u(G) \leqs 6$. By combining the proof of Theorem \ref{t:ExceptionalUDN} with \cite[Theorem 2]{BLS}, we deduce that there is an element $s \in G$ such that the probability that $6$ randomly chosen conjugates of $s$ form a total dominating set for $\Gamma(G)$ tends to $1$ as $|G| \to \infty$.
\end{remk}

As noted above, in order to effectively apply the probabilistic approach, we need to find an element $s \in G^{\#}$ that is contained in a small number of maximal subgroups of $G$. In this way, it is natural to consider the parameter
\[
\mu(G) = \min_{s \in G}|\M(G,s)|.
\]
We establish the following result for simple groups.

\begin{theorem}\label{t:main5}
If $G$ is a finite simple group, then either $\mu(G) \leqs 3$ or $(G,\mu(G))$ is one of the following:
\[
{\renewcommand{\arraystretch}{1.1}
\begin{array}{lcccc}
\hline
G      & {\rm U}_{6}(2) & {\rm U}_4(3) & \O_8^+(2) & {\rm P\O}_8^+(3) \\
\mu(G) & 4              & 5            & 7         & 7                \\
\hline
\end{array}}
\]
In particular, $\mu(G) \leqs 7$ for every finite simple group $G$.
\end{theorem}

In fact, we can compute the exact value of $\mu(G)$ for any alternating or sporadic group $G$ (see Theorems \ref{t:AltMax} and \ref{t:SporadicMax}), and it is worth noting that there are infinitely many alternating groups $G$ with $\mu(G)=3$. The result for alternating and classical groups is essentially a corollary of the proof of the main theorem of Guralnick and Kantor \cite{GK}, which identifies an explicit element that is contained in very few maximal subgroups (typically, this is a regular semisimple element when $G$ is a classical group). Finally, we appeal to earlier work of Weigel \cite{Wei}, which shows that almost every finite simple exceptional group of Lie type has an element that is contained in a unique maximal subgroup.

\vs

Our notation is standard. We adopt the notation from \cite{KL} for simple groups, so we write ${\rm L}_{n}(q) = {\rm PSL}_{n}(q)$ and ${\rm U}_{n}(q) = {\rm PSU}_{n}(q)$ for linear and unitary groups, and ${\rm P\O}_{n}^{\e}(q)$ is a simple orthogonal group, etc. In addition, we will write $(a_1, \ldots, a_k)$ for the greatest common divisor of a collection of positive integers $a_1, \ldots, a_k$. 


\section{Methods}\label{s:prel}

In this section we introduce some of the main tools that will be needed in the proofs of Theorems~\ref{t:main1}--\ref{t:main5}. First, in Section~\ref{ss:base}, we establish an important connection between the uniform domination number and base sizes. Our probabilistic approach to bounding the uniform domination number, based on fixed point ratio estimates, is presented in Section~\ref{ss:prob}; here the main result is Lemma~\ref{l:ProbMethod}. Finally, in Section~\ref{ss:comp}, we outline some of the computational methods that we will employ.

\subsection{Bases}\label{ss:base}

Let $G$ be a finite group with generating graph $\Gamma(G)$. Recall that a subset $S \subseteq G^{\#}$ is a \emph{total dominating set} (TDS for short) for $\Gamma(G)$ if for all $g \in G^{\#}$ there exists $s \in S$ such that $G = \la g, s \ra$. For any $g \in G$, write $\M(G,g)$ for the set of maximal subgroups of $G$ containing $g$.

\begin{lem}\label{l:Criterion}
A subset $\{s_1,\ldots,s_c\} \subseteq G^{\#}$ is a total dominating set for $\Gamma(G)$ if and only if 
\[ 
\bigcap_{i=1}^{c} H_i = 1 
\]
for all $(H_1,\ldots,H_c) \in \prod_{i=1}^{c} \M(G,s_i)$. 
\end{lem}

\begin{proof}
Let $S = \{ s_1,\ldots,s_c \}$. By definition, $S$ is not a total dominating set if and only if there exists $g \in G^{\#}$ such that $G \ne \<g,s_i\>$ for all $i$; that is, for each $i$, $g$ is contained a maximal subgroup of $G$ containing $s_i$. So $S$ is not a total dominating set if and only if there exists $(H_1,\ldots,H_c) \in \prod_{i} \M(G,s_i)$ such that $\bigcap_{i} H_i \neq 1$.  The result follows.
\end{proof}

Let $G$ be a group acting faithfully on a finite set $\O$. Recall that a subset $B \subseteq \O$ is a \emph{base} if the pointwise stabiliser of $B$ in $G$ is trivial. Write $b(G,\O)$ for the minimal size of a base for the action of $G$ on $\O$. Note that if $G$ is transitive on $\Omega$ and $H$ is a point stabiliser, then $b(G,\O) \leqs c$ if and only if there exist $g_1, \ldots, g_c \in G$ such that 
\[ 
\bigcap_{i=1}^{c} H^{g_i} = 1. 
\] 
Let us also observe that 
\begin{equation}\label{e:lb}
b(G,\O) \geqs \frac{\log|G|}{\log|\O|}.
\end{equation}

\begin{cor}\label{c:CriterionBase}
Suppose there is an element $s \in G$ such that $\M(G,s) = \{H\}$ and $H$ is core-free. Then $b(G,G/H)$ is the minimal size of a total dominating set for $\Gamma(G)$ containing only conjugates of $s$.
\end{cor}

\begin{proof}
Let $c$ be a positive integer. As noted above, $b(G,G/H) \leqs c$ if and only if there exist $g_1, \ldots, g_c \in G$ such that $\bigcap_{i} H^{g_i} = 1$. Since $\M(G,s^{g_i}) = \{ H^{g_i} \}$ for each $i$, the result follows from Lemma~\ref{l:Criterion}.
\end{proof}

This corollary connects the study of base sizes for primitive permutation groups to the existence of total dominating sets comprising conjugate elements. In particular, it leads us naturally to the notion of the \emph{uniform domination number} $\gamma_u(G)$ of $G$ introduced in Section~\ref{s:intro}, which is our main focus in this paper. Recall that this is defined to be the minimal number of conjugate elements that form a total dominating set for $\Gamma(G)$. 

The main goal of this paper is to study $\gamma_u(G)$ for finite simple groups $G$. (By the main theorem of \cite{GK}, simple groups have uniform spread $1$ and thus $\gamma_u(G)$ is well-defined.) Of course, in this situation every proper subgroup of $G$ is core-free and so we are in a position to apply Corollary~\ref{c:CriterionBase}. Indeed, if we can identify an element $s \in G$ with $\M(G,s) = \{H\}$ then $\gamma_u(G) \leqs b(G,G/H)$ and we can appeal to the extensive literature on bases for simple groups, which is a topic that has seen a great deal of activity in recent years (see \cite{B07,BGS,BLS,BOW,Hal}, for example). In the next section, we will develop a probabilistic approach which can be used to obtain upper bounds on $\gamma_u(G)$ in the general case where we have an element $s \in G$ that is contained in several maximal subgroups. 

We conclude this section by recording a result which allows us to exploit information on base sizes to determine lower bounds on $\gamma_u(G)$.

\begin{cor}\label{c:LowerBound}
Let $s \in G^{\#}$ and let $H \in \M(G,s)$ with $b(G,G/H) = b$. Then any total dominating set for $\Gamma(G)$ containing only conjugates of $s$ has size at least $b$.
\end{cor}

\begin{proof}
Let $\{s^{g_1}, \ldots, s^{g_c}\}$ be a total dominating set for $\Gamma(G)$. Then $H^{g_i} \in \M(G,s^{g_i})$ for all $i$. Therefore, Lemma~\ref{l:Criterion} implies that $\bigcap_{i}H^{g_i} = 1$ and thus $b \leqs c$.
\end{proof}

\begin{rem}\label{r:LowerBound}
We can use Corollary~\ref{c:LowerBound} to derive lower bounds on $\gamma_u(G)$. Indeed, if there is a positive integer $c$ such that for each $s \in G^{\#}$ there exists $H \in \M(G,s)$ with $b(G,G/H) \geqs c$, then $\gamma_u(G) \geqs c$. Of course, one only needs to check this condition on $s$ for a set of conjugacy class representatives. In fact, it suffices only to check for a set $\{g_1, \ldots, g_m\}$ of class representatives with the property that for all $x \in G^{\#}$ there exists $y \in g_1^G \cup \cdots \cup g_m^G$ such that $x = y^{\ell}$ for some integer $\ell$.
\end{rem}

\subsection{Probabilistic methods}\label{ss:prob}

In \cite{LSh99}, Liebeck and Shalev introduced a probabilistic approach for studying the base size of a finite transitive permutation group $G \leqs {\rm Sym}(\O)$. The basic idea is to consider the probability
that a randomly chosen $c$-tuple of points in $\O$ is \emph{not} a base for $G$ and then show that this probability is strictly less than $1$ for some
appropriate positive integer $c$; this immediately implies that $b(G,\O)
\leqs c$. This has proven to be an effective way of establishing accurate (upper) bounds on the base size of almost simple primitive permutation groups; indeed, this is the main tool in the proof of an influential conjecture of Cameron on the base size of so-called \emph{non-standard} primitive groups (see \cite{BLS} and the references therein). Here our goal is to develop a similar approach to study the uniform domination number of simple groups.

Let $G$ be a finite group, let $c$ be a positive integer and fix an element $s \in G^{\#}$. Write $Q(G,s,c)$ for the probability that a random $c$-tuple $(z_1, \ldots, z_c)$ of conjugates of $s$ is such that $\{ z_1, \ldots, z_c \}$ is \emph{not} a total dominating set for $\Gamma(G)$. Consequently, $\gamma_u(G) \leqs c$ if $Q(G,s,c) < 1$ for some $s$. In order to present an upper bound for $Q(G,s,c)$, we need some additional notation. For an element $x \in G$ and a subgroup $H < G$, let 
\begin{equation}\label{e:fpr}
\fpr(x,G/H) = \frac{|x^G \cap H|}{|x^G|}
\end{equation}
be the \emph{fixed point ratio} of $x$ in the action of $G$ on the set of cosets $G/H$; that is, $\fpr(x,G/H)$ is the proportion of points in $G/H$ fixed by $x$ (equivalently, it is the probability that a randomly chosen coset of $H$ is fixed by $x$). Let $\{x_1,\ldots, x_k\}$ be a set of representatives of the conjugacy classes in $G$ of prime order elements. 

We can now present our key lemma for studying uniform domination numbers.

\begin{lem}\label{l:ProbMethod}
Let $G$ be a finite group, $s \in G^{\#}$ and $c \in \Nat$. Then 
\begin{equation}\label{e:qsc}
Q(G,s,c) \leqs \sum_{i=1}^{k} |x_i^G| \left( \sum_{H \in \M(G,s)}^{} \fpr(x_i,G/H) \right)^c \eqqcolon \widehat{Q}(G,s,c).
\end{equation}
\end{lem}

\begin{proof}
Let $\mathcal{P}$ be the set of elements in $G$ of prime order.  For each $x \in G$, let 
\[
P(x,s) = \frac{|\{ z \in s^G \,:\, G \ne \< x, z \>\}|}{|s^G|} = \frac{|\{ g \in G \,:\, G \ne \< x, s^g \>\}|}{|G|}
\]
be the probability that a randomly chosen conjugate of $s$ does not generate $G$ with $x$. Since $G \ne \< x, s^g \>$ if and only if $x^{g^{-1}} \in H$ for some $H \in \mathcal{M}(G,s)$, it follows that 
\[
P(x,s) \leqs \sum_{H \in \mathcal{M}(G,s)} \frac{|x^G \cap H|\,|C_G(x)|}{|G|} = \sum_{H \in \mathcal{M}(G,s)} {\rm fpr}(x,G/H).
\]
Now $\{s^{g_1}, \ldots, s^{g_c}\}$ is not a total dominating set for $\Gamma(G)$ if and only if there exists 
$x \in \mathcal{P}$ such that $G \ne \<x,s^{g_i}\>$ for all $i$. Therefore,
\begin{equation}\label{e:pxs}
Q(G,s,c) \leqs \sum_{x \in \mathcal{P}} P(x,s)^c = \sum_{i=1}^{k} |x_i^G|\, P(x_i,s)^c
\end{equation}
and the result follows.     
\end{proof}

\begin{rem}
Note that if $\M(G,s)=\{H\}$, then the probabilistic approach via Lemma~\ref{l:ProbMethod} coincides with the method introduced by Liebeck and Shalev in \cite{LSh99} to study the base size $b(G,G/H)$. Accordingly, Lemma~\ref{l:ProbMethod} gives no more information than Corollary~\ref{c:CriterionBase} in this case. However, the utility of Lemma~\ref{l:ProbMethod} is that it allows us to handle groups for which every element belongs to at least two maximal subgroups.
\end{rem}

The following elementary observation is a natural extension of \cite[Proposition 2.3]{BLS}. 

\begin{lem}\label{l:bd}
Let $G$ be a finite group and let $\{H_1, \ldots, H_{\ell}\}$ be proper subgroups of $G$. Suppose that $x_1, \ldots, x_m$ represent distinct $G$-classes such that $\sum_{i}|x_i^G \cap H_j| \leqs A_j$ and $|x_i^G| \geqs B$ for all $i,j$. Then
\[
\sum_{i=1}^{m}|x_i^G|\left(\sum_{j=1}^{\ell}{\rm fpr}(x_i,G/H_j)\right)^c \leqs B^{1-c}\left(\sum_{j}A_j\right)^c
\]
for all $c \in \mathbb{N}$.
\end{lem}

\begin{proof}
Write $a_{ij} = |x_i^G \cap H_j|$ and $b_i = |x_i^G|$, so $\sum_{i}a_{ij} \leqs A_j$ and $b_i \geqs B$. Then the left hand side of the required inequality is
\[
\sum_{i}b_i\left(\sum_j a_{ij}/b_i\right)^c = \sum_{i} b_{i}^{1-c}\left(\sum_j a_{ij}\right)^c \leqs B^{1-c}\sum_{i}\left(\sum_j a_{ij}\right)^c \leqs B^{1-c}\left(\sum_{i,j} a_{ij}\right)^c
\]
and the result follows.
\end{proof}

It is natural to expect that the upper bound in \eqref{e:qsc} will be easier to compute if we can find an element $s \in G$ such that $|\M(G,s)|$ is small. With this in mind, it is interesting to study the following parameter
\begin{equation}\label{e:mu}
\mu(G) = \min_{s \in G}|\M(G,s)|,
\end{equation} 
which we introduced in Section~\ref{s:intro}. Our main result is Theorem~\ref{t:main5}, which reveals that every finite simple group has an element that is contained in very few maximal subgroups (at most $3$ in fact, apart from $4$ specific exceptions). 

\subsection{Computational methods}\label{ss:comp}

For some small simple groups $G$, we can use computational methods, implemented in \textsf{GAP} \cite{GAP} and \textsc{Magma} \cite{magma}, to study $\mu(G)$ and $\gamma_u(G)$. For example, all of our results for sporadic groups (see Section~\ref{s:spor}) are obtained by computation. Here we outline the main techniques. 

A detailed description of these computations can be found at \cite{BH_comp}, including the relevant \textsf{GAP} and {\sc Magma} code we used to obtain the results.

\subsubsection{Probabilistic methods}\label{sss:comp_prob}

Let us first describe an implementation of the probabilistic method introduced in Section~\ref{ss:prob}. Let $G$ be a finite group and fix an element $s \in G^{\#}$. Our aim is to determine the minimal value of $c$ such that $\what{Q}(G,s,c) < 1$ (see Lemma~\ref{l:ProbMethod}). In order to calculate $\what{Q}(G,s,c)$ we first need to determine $\M(G,s)$, and then calculate the fixed point ratios $\fpr(x,G/H)$ for each prime order element $x \in G$ and subgroup $H \in \M(G,s)$. If such a subgroup $H$ is self-normalising, then $s$ is contained in exactly 
\[
\fpr(s,G/H)\cdot|G:H|
\] 
distinct conjugates of $H$. Therefore, if $G$ is simple then in order to determine $\M(G,s)$ it suffices to compute $\fpr(s,G/H)$ for a representative $H$ of each conjugacy class of maximal subgroups of $G$. Hence, we focus on determining the fixed point ratios $\fpr(x,G/H)$ for elements $x \in G$ and maximal subgroups $H < G$. Of course, this approach via \textsf{GAP} and \textsc{Magma} is only feasible if $G$ is amenable to computational methods, which typically means that the order of $G$, or the minimal degree of a faithful permutation representation, is not too large.

If the Character Table Library \cite{CTblLib} in \textsf{GAP} contains the ordinary character tables of $G$ and each of its maximal subgroups, then we can adopt the techniques of Breuer, which are detailed in \cite[Section 3.2]{Breuer}. Indeed, in this situation we can compute $\fpr(x,G/H)$ by observing that the number of fixed points of $x$ on $G/H$ is equal to $\chi(x)$, where $\chi=1^G_H$ is the corresponding permutation character. If this character-theoretic approach is not available, then we turn to \textsc{Magma}. If the functions \texttt{MaximalSubgroups} and \texttt{Classes} return the maximal subgroups and conjugacy classes of $G$, then we can calculate $\fpr(x,G/H)$ via \eqref{e:fpr}, using \texttt{IsConjugate} to compute $|x^G \cap H|$. 

\subsubsection{Maximal subgroups}\label{sss:comp_max}

We can often use the above methods to determine the maximal overgroups of a specific element of $G$, which allows us to determine $\mu(G)$ in this way. In addition, by determining $\M(G,s)$ for a complete set of conjugacy class representatives $s$, we can apply the observations in Remark~\ref{r:LowerBound} to derive a lower bound on $\gamma_u(G)$.

\subsubsection{Random and exhaustive searches}\label{sss:comp_search}

Let $c \geqs 2$ be an integer and fix $s \in G^{\#}$. If we have a faithful permutation representation of $G$, which permits calculation in {\sc Magma}, then we can randomly choose $c$-tuples of conjugates of $s$ and check whether they form a TDS for $\Gamma(G)$. Of course, if we find such a $c$-tuple, then $\gamma_u(G) \leqs c$. In contrast, to establish the bound $\gamma_u(G) > c$ we must show that for each conjugacy class $s^G$, there are no $c$-tuples of conjugates of $s$ which form a TDS. Here it is helpful to observe that if there is such a $c$-tuple, then there is one containing $s$. Therefore, when trying to verify upper (or lower) bounds on $\gamma_u(G)$ we may randomly (or exhaustively) choose $(c-1)$-tuples of elements in $s^G$ and check whether they, together with $s$, form a TDS for $\Gamma(G)$.

\begin{rem}\label{r:computation}
Let us say a few words on the computational resources needed for the main calculations. For the computations in this paper, we use a combination of \textsf{GAP} Version 4.5.6 and {\sc Magma} 2.19-2, on a 2.7GHz machine with 128\,GB RAM. The character-theoretic computations run quickly in \textsf{GAP} and we adopt this approach whenever possible. The computations in \textsc{Magma} for determining maximal overgroups of specific elements and implementing the probabilistic approach (via fixed point ratios) are more resource-intensive, but still feasible for the groups we are interested in. For example, an implementation of the probabilistic method applied to ${\rm L} _9(2)$ with an element of order $465$ (see Proposition~\ref{p:bdc}) can be done in $616$ seconds, using $771\,\text{MB}$ of memory. Similar resources are needed for most of the exhaustive searches in \textsc{Magma}, which we use to rule out the existence of total domination sets with prescribed properties. However, the verification of the bound $\gamma_u({\rm M}_{12}) \geqs 4$ is a notable exception. Here, in view of Corollary \ref{c:LowerBound} and the base size results in \cite{BOW}, we quickly reduce the problem to showing that no triple of elements in the class \texttt{6A} form a TDS for $\Gamma({\rm M}_{12})$. We timed this computation at $17613$ seconds, using $13\,\text{MB}$ of memory.
\end{rem}

For the remainder of the paper we will focus on the proofs of Theorems~\ref{t:main1}--\ref{t:main5}, by considering each family of (non-abelian) simple groups $G$ in turn. In each case, we first study the parameter $\mu(G)$ defined in \eqref{e:mu}, with the aim of establishing a strong form of Theorem~\ref{t:main5}. We then establish our main results on the uniform domination number of $G$.


\section{Alternating groups}\label{s:alt}

\subsection{Maximal overgroups}\label{ss:alt_max}

In this section we verify Theorem~\ref{t:main5} for alternating groups. More precisely, we compute the exact value of $\mu(G)$ for each simple alternating group $G$. In order to state our main result, let 
\begin{equation}\label{e:H}
\mathcal{H} = \left\{ n \in \mathbb{N}\,:\, \mbox{$n =\frac{q^d-1}{q-1}$ for some prime power $q$ and integer $d \geqs 2$} \right\}.
\end{equation}
We refer the reader to \cite[Table II]{BS} for a convenient list of the first $240$ primes in $\mathcal{H}$.

\begin{thm}\label{t:AltMax}
Let $G=A_n$ with $n \geqs 5$. 
Then $\mu(G) \leqs 3$. Moreover,
\begin{itemize}\addtolength{\itemsep}{0.2\baselineskip}
\item[{\rm (i)}] $\mu(G) = 1$ if and only if one of the following hold:
\begin{itemize}\addtolength{\itemsep}{0.2\baselineskip}
\item[{\rm (a)}] $n=5$;
\item[{\rm (b)}] $n \geqs 8$ is even;
\item[{\rm (c)}] $n \in \{r,r^2\}$, where $r$ is a prime, $n \not\in \{11,23\}$ and $n \not\in \mathcal{H}$.
\end{itemize}
\item[{\rm (ii)}] $\mu(G) = 2$ if and only if one of the following hold:
\begin{itemize}\addtolength{\itemsep}{0.2\baselineskip} 
\item[{\rm (a)}] $n \in \{6, 7, 11, 17, 23\}$; 
\item[{\rm (b)}] $n \in \{rs,r^3\}$, where $r,s$ are distinct odd primes and $n \not\in \mathcal{H}$.
\end{itemize}
\end{itemize}
In particular, Theorem~\ref{t:main5} holds for alternating groups.
\end{thm}

The proof of Theorem~\ref{t:AltMax} closely follows the proof of \cite[Proposition 7.1]{GK}, which identifies an element $g \in G$ that is contained in very few maximal subgroups. Indeed, the bound $\mu(G) \leqs 3$ is an immediate corollary of the proof of \cite[Proposition 7.1]{GK} (with a small correction when $n$ is an odd integer of the form $3m$) but more work is needed to compute the exact value in every case.

\begin{rem}\label{r:alt}
There are infinitely many primes $r$ with $\mu(A_r)=1$. To see this, we need to show that there are infinitely many prime numbers that are not contained in $\mathcal{H}$. For a real number $x$, let $\pi(x)$ be the number of primes less than or equal to $x$, and let $H(x)$ be the number of primes at most $x$ in $\mathcal{H}$. By the prime number theorem, we have 
$\pi(x) = (1+o(1))x(\log x)^{-1}$, whereas \cite[Theorem 4]{BS} gives
$H(x) \leqs 50x^{1/2}(\log x)^{-2}$ for $x \gg 0$. In other words, if $x$ is large enough then almost all primes at most $x$ are not in $\mathcal{H}$.
\end{rem}

In order to prove Theorem~\ref{t:AltMax}, we need to record some preliminary lemmas. The first follows from the main theorem in \cite{LPS}.

\begin{lem}\label{l:lps}
Let $G = A_n$ with $n \geqs 5$ and let $H$ be an intransitive subgroup of the form $(S_k \times S_{n-k}) \cap G$ with $k < n/2$, or an imprimitive subgroup $(S_k \wr S_{n/k}) \cap G$ with $1<k \leqs n/2$. Then $H$ is a maximal subgroup of $G$ unless $G = A_8$ and $H = (S_2 \wr S_4) \cap G$. 
\end{lem}

We denote the shape of a permutation $g \in S_n$ by writing $[l_1, \ldots, l_t]$ with $\sum_{i}l_i = n$, where the $l_i$ are the lengths of the disjoint cycles comprising $g$ (in addition, if $g$ has $b_i$ cycles of length $a_i$, where $a_1>a_2 > \cdots > a_k$, then it will be convenient to write $[a_1^{b_1}, \ldots, a_k^{b_k}]$ for the shape of $g$). The next result concerns the containment of certain elements in imprimitive subgroups; the proof is a straightforward application of \cite[Theorem 2.5]{AAC}. 

\begin{lem}\label{lem:Imprimitive}
Let $G = S_n$, with $n \geqs 5$, and let $g \in G$. 
\begin{itemize}\addtolength{\itemsep}{0.2\baselineskip}
\item[{\rm (i)}] If $g$ is an $n$-cycle, then $g$ is not contained in an imprimitive subgroup of $G$ if and only if $n$ is a prime. Moreover, if $n=mk$ with $m,k>1$, then $g$ is contained in a unique subgroup $S_m \wr S_k$.
\item[{\rm (ii)}] If $g$ has shape $[l_1,l_2]$, then $g$ is not contained in an imprimitive subgroup of $G$ if and only if $(l_1,l_2)=1$.
\item[{\rm (iii)}] If $g$ has shape $[l_1,l_2,l_3]$, then $g$ is not contained in an imprimitive subgroup of $G$ if and only if $(l_1,l_2,l_3) = 1$, and if $d$ divides $(l_i,l_j)$ then $\frac{l_i+l_j}{d}$ does not divide $l_k$, where $\{i,j,k\} = \{1,2,3\}$.
\end{itemize}
\end{lem}

The following lemma is a classical result of Marggraf (see \cite[Theorem 13.5]{W}).

\begin{lem}\label{lem:Marggraf}
If a primitive subgroup $H \leqs S_n$ contains a cycle of length $\ell < n/2$, then $H = A_n$ or $S_n$. 
\end{lem}

We will also need the following technical result.

\begin{lem}\label{l:pgam}
Let $G=A_n$, where $n \geqs 7$ is odd. Suppose $n = (q^d-1)/(q-1)$, where $d \geqs 2$ and $q=p^f$ for a prime $p$. Let $g \in G$ be an $n$-cycle and let $N$ be the number of subgroups of $G$ of the form ${\rm P\Gamma L}_d(q) \cap G$ containing $g$. Then
\[
\frac{\varphi(n)}{2df} \leqs N \leqs \frac{\varphi(n)}{d},
\]
where $\varphi$ is Euler's totient function.
\end{lem}

\begin{proof}
For $n \in \{7,9\}$, it is easy to check that $N = \varphi(n)/d$, so we may assume $n>9$.
Fix a subgroup $H = {\rm P\Gamma L}_d(q) \cap G$ containing $g$ and let 
$k$ be the number of $G$-conjugates of $H$ containing $g$, so 
\[ k = {\rm fpr}(g,G/H)\cdot |G:H| = \frac{|g^G \cap H|}{|g^G|} \cdot |G:H|.\] 
Since $C_G(g)=C_H(g)=\<g\>$, it follows that $k$ is the number of $H$-classes in $g^G \cap H$. 

By \cite[Theorem 1]{J}, every $n$-cycle in $H$ generates a Singer subgroup of ${\rm PGL}_d(q)$ and thus $H$ contains $\varphi(n)/d$ distinct ${\rm PGL}_d(q)$-classes of $n$-cycles. Since only half of these classes are contained in $g^G$, we get $k \leqs \varphi(n)/2d$. By considering the fusing action of field automorphisms in ${\rm P\Gamma L}_d(q)$ on ${\rm PGL}_d(q)$-classes, we also deduce that $k \geqs \varphi(n)/2df$. Finally, we note that $S_n$ has a unique class of subgroups of the form ${\rm P\Gamma L}_d(q)$ (the corresponding actions of ${\rm P\Gamma L}_d(q)$ on lines and hyperplanes in $\mathbb{F}_q^d$ are permutation isomorphic), so $G$ contains at most two conjugacy classes of subgroups of the form ${\rm P\Gamma L}_d(q) \cap G$. We conclude that $N \leqs 2k \leqs \varphi(n)/d$.
\end{proof}

We are now ready to prove Theorem~\ref{t:AltMax}.

\begin{proof}[Proof of Theorem~\ref{t:AltMax}]
First assume $n$ is even. If $n=6$, then it is easy to check that $\mu(G)=2$; in particular, if $g \in G$ is a $5$-cycle then $\M(G,g) = \{H, K\}$ with $H \cong K \cong A_5$. Now assume $n \geqs 8$. We claim that $\mu(G)=1$. To see this, write $n=2m$, $k=m-(m-1,2)$ and choose an element $g \in G$ with shape $[k,n-k]$. The unique intransitive subgroup in $\M(G,g)$ has the form $(S_k \times S_{n-k}) \cap G$, and imprimitive groups are ruled out by Lemma~\ref{lem:Imprimitive}(ii) since $(k,n-k)=1$. Furthermore, Lemma~\ref{lem:Marggraf} eliminates primitive subgroups since $g^{n-k}$ is a $k$-cycle and $k < n/2$. Therefore, $\M(G,g) = \{H\}$ with $H = (S_k \times S_{n-k}) \cap G$. 

For the remainder, we may assume $n$ is odd. If $n \leqs 23$ then we verify the result computationally in {\sc Magma} (see Section~\ref{sss:comp_prob}). In particular, the value of $\mu(G)$ and the shape of an element $g$ for which $|\M(G,g)|=\mu(G)$ are as follows: 
\[
\begin{array}{lcccccccccc}
\hline
n      &   5 &   7 &       9 &   11 &      13 &       15 &   17 &   19 &       21 &   23 \\
\mu(G) &   1 &   2 &       3 &    2 &       3 &        3 &    2 &    1 &        3 &    2 \\
g      & [5] & [7] & [5,2^2] & [11] & [9,2^2] & [11,2^2] & [17] & [19] & [17,2^2] & [23] \\
\hline
\end{array}
\] 

Now assume $n \geqs 25$ is odd. Our goal is to establish the following five statements:
\begin{itemize}\addtolength{\itemsep}{0.2\baselineskip} 
\item[(1)] $\mu(G) \leqs 3$.
\item[(2)] If $n \in \{r,r^2\}$, where $r$ is a prime and $n \not\in \mathcal{H}$, then $\mu(G) = 1$.
\item[(3)] If $n \in \{rs, r^3\}$, where $r,s$ are distinct primes and $n \not\in \mathcal{H}$, then $\mu(G) \leqs 2$.
\item[(4)] If $\mu(G) \leqs 2$ then $n \not\in \mathcal{H}$ and $n \in \{r,r^2,r^3,rs\}$ for distinct primes $r,s$.
\item[(5)] If $n \in \{rs,r^3\}$, where $r,s$ are distinct primes, then $\mu(G) \geqs 2$.
\end{itemize}
Indeed, observe that (1)--(5) complete the proof of Theorem \ref{t:AltMax}. More precisely, (1) gives the main statement of the theorem, (2) completes the reverse implication in part (i), (3) gives the reverse implication of (ii), and by combining (2), (4) and (5) we obtain the forward implications in parts (i) and (ii).

First consider (1). Let $g \in G$ be an element with the following shape:
\begin{equation}\label{e:s}
\left\{\begin{array}{ll}
\mbox{$[m+2,m,m-2]$}   & \mbox{if $n=3m$} \\
\mbox{$[m+1,m+1,m-1]$} & \mbox{if $n=3m+1$} \\
\mbox{$[m+2,m,m]$}     & \mbox{if $n=3m+2$.} 
\end{array}\right.
\end{equation}
By applying Lemma~\ref{lem:Imprimitive}(iii), we deduce that $g$ does not have any imprimitive maximal overgroups. For example, if $n=3m+1$ then $m$ is even so $(m+1,m-1)=1$. Moreover, $(m+1+m-1)/1=2m$ does not divide $m+1$, and $(m+1+m+1)/(m+1) = 2$ does not divide $m-1$. The other two cases are similar. As before, primitive maximal overgroups can be ruled out via Lemma~\ref{lem:Marggraf} and we conclude that $\mu(G) \leqs 3$ as claimed.

Now consider (2), so $n \in \{r,r^2\}$ and $n \not\in \mathcal{H}$ for a prime $r$. Let $g \in G$ be an $n$-cycle and note that $g$ is not contained in an intransitive subgroup. Suppose $n=r$. By Lemma~\ref{lem:Imprimitive}(i), 
$g$ is not contained in an imprimitive subgroup. Moreover, \cite[Theorem~3]{J} implies that ${\rm AGL}_1(r) \cap G$ is the only primitive maximal overgroup of $g$, so $\mu(G)=1$ as required. Similarly, if $n=r^2$ then by applying \cite[Theorem~3]{J} to rule out primitive groups we deduce that $g$ is contained in a unique imprimitive subgroup $(S_r \wr S_r) \cap G$, so $\mu(G)=1$ once again.

Next consider (3) and let $g$ be an $n$-cycle. As before, by applying \cite[Theorem~3]{J}, we deduce that the maximal overgroups of $g$ are imprimitive. More precisely,
\[
\M(G,g) = \left\{\begin{array}{ll}
\{ (S_r \wr S_s) \cap G, (S_s \wr S_r) \cap G \}         & \mbox{if $n=rs$} \\
\{ (S_r \wr S_{r^2}) \cap G, (S_{r^2} \wr S_r) \cap G \} & \mbox{if $n=r^3$}
\end{array}\right.
\]
and thus $\mu(G) \leqs 2$.

Let us now turn to (4). Suppose that $\mu(G) \leqs 2$ and fix $g \in G$ with  $|\M(G,g)|=\mu(G)$. Since $g$ is even and $n$ is odd, $g$ is not the product of exactly two cycles. If $g$ has at least three cycles, then $\M(G,g)$ contains at least three intransitive subgroups, so $g$ must be an $n$-cycle. If $n$ has at least three distinct prime divisors, or if $n=r^2s$ for distinct primes $r,s$, then Lemma~\ref{lem:Imprimitive}(i) implies that $\M(G,g)$ contains at least three imprimitive subgroups. Therefore, $n \in \{r,r^2,r^3,rs\}$, where $r,s$ are distinct primes, and it remains to prove that $n \not\in \mathcal{H}$. If $n \in \{r^3, rs\}$, then $g$ is already contained in two imprimitive subgroups, so the condition $\mu(G) \leqs 2$ implies that $g$ is not contained in a proper primitive subgroup, whence $n \not\in \mathcal{H}$ by 
\cite[Theorem~3]{J}. The cases $n \in \{r,r^2\}$ require special attention.

First assume $n=r$. Seeking a contradiction, suppose that $n \in \mathcal{H}$. 
By \cite[Theorem~3]{J}, $\M(G,g)$ contains a unique subgroup of the form ${\rm AGL}_1(r) \cap G$ (namely, $N_G(\<g\>)$), together with a collection of subgroups ${\rm P\Gamma L}_d(q)$ for each prime power $q$ and integer $d \geqs 2$ such that $n=(q^d-1)/(q-1)$. More precisely, for each $(q,d)$ with $q=p^f$ and $p$ a prime, Lemma~\ref{l:pgam} implies that $g$ is contained in at least $(n-1)/2df$ such subgroups. Since $n \geqs 25$, one can check that this gives $|\M(G,g)|>2$ and we have reached a contradiction. A similar argument applies if $n=r^2 \in \mathcal{H}$. Indeed, $\M(G,g)$ contains a unique subgroup of the form $(S_r \wr S_r) \cap G$ and at least two primitive subgroups ${\rm P\Gamma L}_d(q) \cap G$. Once again, this is a contradiction and the proof of (4) is complete. 

Finally, let us consider (5), so $n \in \{rs,r^3\}$ for distinct primes $r$ and $s$. Fix an element $1 \ne g \in G$. If $g$ is not an $n$-cycle then the argument in (4) shows that $g$ is contained in at least three intransitive maximal subgroups. On the other hand, if $g$ is an $n$-cycle then the proof of (3) implies that $g$ is contained in at least two imprimitive subgroups. Therefore, $|\M(G,g)| \geqs 2$ and thus $\mu(G) \geqs 2$. 
\end{proof}

\subsection{Uniform domination number}\label{ss:alt_udn}

We now apply Theorem~\ref{t:AltMax} to study the uniform domination number of alternating groups. Our main result is the following.

\begin{thm}\label{t:mainn}
We have $\gamma_u(A_n) \leqs 77\log_2n$ for all $n \geqs 5$. More precisely, the following hold:
\begin{itemize}\addtolength{\itemsep}{0.2\baselineskip}
\item[{\rm (i)}] If $n \geqs 13$ is a prime, then $\gamma_u(A_n) = 2$.
\item[{\rm (ii)}] If $n \geqs 6$ is even, then $\lceil\log_2 n \rceil - 1 \leqs \gamma_u(A_n) \leqs 2\lceil\log_2 n \rceil$.
\end{itemize}
\end{thm}

We partition the proof of Theorem~\ref{t:mainn} into three propositions which are proved in the following three sections.

\subsubsection{Prime degree}
We start by considering alternating groups of prime degree.

\begin{prop}\label{t:main_alt}
Let $r \geqs 13$ be a prime number. Then $\gamma_u(A_r)=2$.
\end{prop}

We need some preliminary lemmas. Recall the definition of the set of integers $\mathcal{H}$ in \eqref{e:H}. Fix a prime number $r \geqs 5$ and set
\[
\mathcal{H}_r = \left\{ (q,d) \,:\, \mbox{$r =\frac{q^d-1}{q-1}$, $q$ a prime power, $d \geqs 2$} \right\}.
\]

\begin{lem}\label{l:hp}
$|\mathcal{H}_r| < \log_2r$.
\end{lem}

\begin{proof}
Suppose $(q,d) \in \mathcal{H}_r$. Then $d$ is a prime number and it is easy to see that $(s,d) \in \mathcal{H}_r$ if and only if $s=q$. Indeed, if $q<s$ then 
\[
\frac{q^d-1}{q-1} = \frac{s^d-1}{s-1} \geqs \frac{(q+1)^d-1}{q},
\]
which is absurd. Therefore, we just need to count the possibilities for $d$.  Since $r>q^{d-1}$ we have
\[
d \leqs \left\lfloor \frac{\log r}{\log q}\right\rfloor +1=:D
\]
and the result follows since there are fewer than $\log_2r$ primes at most $D$.
\end{proof}

As an immediate corollary, it follows that if $r \geqs 5$ is a prime in $\mathcal{H}$ then
\begin{equation}\label{e:ell}
\ell\coloneqq1+\sum_{(q,d) \in \mathcal{H}_r}\frac{r-1}{d} < r\log_2r.
\end{equation}

Set $G = A_r$, where $r \geqs 73$ is a prime in $\mathcal{H}$. Fix an $r$-cycle $s \in G$. As observed in the proof of Theorem~\ref{t:AltMax}, $\mathcal{M}(G,s)$ comprises a single copy of ${\rm AGL}_{1}(r) \cap G$, together with at most $(r-1)/d$ copies of ${\rm P\Gamma L}_{d}(q) \cap G$ for each $(q,d) \in \mathcal{H}_r$. In particular, $|\mathcal{M}(G,s)| \leqs \ell$, where $\ell$ is given in \eqref{e:ell}. 

\begin{lem}\label{l:bd0}
If $r \geqs 73$ is a prime in $\mathcal{H}$, then $|H| \geqs r(r-1)/2$ for all $H \in \mathcal{M}(G,s)$.
\end{lem}

\begin{proof}
First observe that $|{\rm AGL}_{1}(r) \cap G| \geqs r(r-1)/2$. Fix $(q,d) \in \mathcal{H}_r$ and note that $r<2q^{d-1}$. Then
\[
|{\rm P\Gamma L}_{d}(q) \cap G| > \frac{1}{4}q^{d^2-1} >\frac{1}{4}\left(\frac{r}{2}\right)^{d+1} \geqs \frac{1}{32}r^3
\]
and the result follows.
\end{proof}

Define 
\begin{equation}\label{e:Cdef}
C = \max\{|H| \,:\, H \in \mathcal{M}(G,s)\}.
\end{equation}

\begin{lem}\label{l:bd2}
Suppose $r \geqs 73$ is a prime in $\mathcal{H}$, $H \in \mathcal{M}(G,s)$ and $x \in H$ has prime order. Then $|x^G|> C^4$.
\end{lem}

\begin{proof}
We use some of the ideas in the proof of \cite[Theorem~1.1]{BGS}. Fix $H \in \mathcal{M}(G,s)$ containing $x$ and observe that $H$ is primitive, so a theorem of Mar\'{o}ti \cite{Mar} gives
\[
|H| < r^{1+\log_2r}.
\]
By the main theorem of \cite{GMag}, the minimal degree of $H$ is at least $r/2$, which means that 
\[
|x^G| \geqs \frac{r!}{2^{r/4}\lceil r/4\rceil!\lceil r/2\rceil!}=:f(r).
\]
Using the bounds
\[
\sqrt{2\pi} \cdot n^{1/2}\left(\frac{n}{e}\right)^n < n! < e\cdot n^{1/2}\left(\frac{n}{e}\right)^n,
\]
which are valid for all positive integers $n$, we get
\begin{align*}
f(r) & > \frac{\sqrt{2\pi} \cdot r^{r+1/2}e^{-r}}{2^{r/4} \cdot e\left(\frac{r+3}{4}\right)^{(r+3)/4+1/2}e^{-(r+3)/4} \cdot e\left(\frac{r+1}{2}\right)^{(r+1)/2+1/2}e^{-(r+1)/2}} \\
     & > 2^{r/4}\left(\frac{r^{r+1/2}}{(r+3)^{r/4+5/4}(r+1)^{r/2+1}}\right) \\
     & > r^{r/4} 
\end{align*}
and thus it suffices to show that $r \geqs 16(1+\log_2r)$. One checks that this inequality holds if $r>127$, so it just remains to handle the primes $r \in \{73,127\}$ (recall that $r \in \mathcal{H}$). If $r=73$ then $C = |{\rm P\Gamma L}_{3}(8)|$ and similarly $C = |{\rm P\Gamma L}_{7}(2)|$ if $r=127$; in both cases, the desired bound is easily checked using {\sc Magma}.
\end{proof}

We are now in a position to prove Proposition~\ref{t:main_alt}.

\begin{proof}[Proof of Proposition~\ref{t:main_alt}]
Set $G = A_r$, where $r \geqs 13$ is a prime number, and let $s \in G$ be an $r$-cycle. First assume $r \not\in \mathcal{H}$. If $r=23$ then a straightforward computation (see Section~\ref{sss:comp_prob}) shows that $\gamma_u(G)=2$. Now assume $r \ne 23$, in which case the proof of Theorem~\ref{t:AltMax} implies that $\M(G,s) = \{ H \}$ with $H = {\rm AGL}_1(r) \cap G$. Now $b(G,G/H) = 2$ by \cite[Corollary~1.5]{BGS}, so Corollary~\ref{c:CriterionBase} implies that $\gamma_u(G) = 2$.

To complete the proof, we may assume $r \in \mathcal{H}$. To handle the case $r=13$, we use random search in {\sc Magma} to show that $\gamma_u(G)=2$ (we refer the reader to \cite[Section~1.2.4]{BH_comp} for further details of the computation). For example, one can check that 
\[
\{ (1,2,3,4,5,6,7,8,9,10,11,12,13), (1,2,3,4,5,6,8,9,12,7,11,10,13) \}
\]
is a TDS for $\Gamma(G)$. For $r \in \{17,31\}$ we can use Lemma~\ref{l:ProbMethod} to show that $\gamma_u(G)=2$. 

Now assume $r>31$, which means that $r \geqs 73$ (see \cite[Table~II]{BS}). Set $\mathcal{M}(G,s) = \{H_1, \ldots, H_{k}\}$ and let $x_1, \ldots, x_m$ be representatives of the $G$-classes of elements of prime order which meet at least one of the subgroups in $\mathcal{M}(G,s)$. Note that $k \leqs \ell$, where $\ell$ is defined in \eqref{e:ell}. By Lemma~\ref{l:ProbMethod}, we need to show that 
\[
\sum_{i=1}^{m}|x_i^G|\left(\sum_{j=1}^{k}{\rm fpr}(x_i,G/H_j)\right)^2 < 1. 
\]
To do this, we apply Lemma~\ref{l:bd} with $A_j = C$ and $B=C^4$ for all $j$, where $C$ is defined in \eqref{e:Cdef}, noting that the value of $B$ is justified by Lemma~\ref{l:bd2}. This yields
\[
\sum_{i=1}^{m}|x_i^G|\left(\sum_{j=1}^{k}{\rm fpr}(x_i,G/H_j)\right)^2 < \left(\frac{\ell}{C}\right)^2.
\]
Finally, we recall that $\ell < r\log_2r$ and $|C| \geqs r(r-1)/2$ (see \eqref{e:ell} and Lemma~\ref{l:bd0}), whence $\ell<C$ and the proof is complete.
\end{proof}

\begin{rem}\label{r:all}
We can use computational methods (see Section~\ref{ss:comp}) to compute $\gamma_u(G)$ when $G=A_r$ and $r \in \{5,7,11\}$ is one of the primes excluded in Proposition~\ref{t:main_alt}. We get the following results:
\[
\gamma_u(A_r) = \left\{ \begin{array}{ll}
3 & \mbox{if $r=5,11$,} \\
4 & \mbox{if $r=7$.}
\end{array}\right.
\]
For example, suppose $r=11$. By applying Lemma~\ref{l:ProbMethod}, with $s \in G$ an $11$-cycle, it is easy to see that $\gamma_u(G) \leqs 3$. To show that $\gamma_u(G) \geqs 3$ we employ the method described in Remark~\ref{r:LowerBound}. Fix an element $g \in G$. If $g$ is not an $11$-cycle then $g$ is contained in an intransitive subgroup $H$ with $b(G,G/H) \geqs 3$ (see Lemma~\ref{thm:Halasi} below). On the other hand, if $g$ is an $11$-cycle then $g$ is contained in a subgroup $H={\rm M}_{11}$ (see the proof of Theorem~\ref{t:AltMax}) and by \cite[Theorem~1]{BGS} we have $b(G,G/H)=3$. Therefore, Corollary~\ref{c:LowerBound} implies that $\gamma_u(G) \geqs 3$ and thus $\gamma_u(G)=3$. The other cases are handled in a similar fashion.
\end{rem}

\subsubsection{Even degree}
Next we consider the alternating groups of even degree; the analysis of the groups of odd composite degree is more complicated and we postpone the study of these groups to the end of the section.

Our main result is Proposition~\ref{t:altbd} below, which gives the exact value of $\gamma_u(A_n)$, up to a small constant. The key tool is the following result of Halasi (see \cite[Theorems 3.1 and 4.2]{Hal}) on the base size of $S_n$ on $k$-sets.

\begin{lem}\label{thm:Halasi}
Let $n \geqs 5$ be an integer and let $\O$ be the set of $k$-element subsets of $\{ 1, \ldots, n \}$ for some $1 \leqs k \leqs n/2$. Then 
\[ 
\left\lceil \log_2 n \right\rceil \leqs b(S_n,\O) \leqs \left\lceil \log_{\lceil n/k \rceil} n \right\rceil \cdot \left( \lceil n/k \rceil - 1 \right). 
\]
\end{lem}

Note that if $S_n$ acts faithfully on a finite set $\O$, then
\[
b(A_n, \O) \leqs b(S_n,\O) \leqs b(A_n, \O) + 1.
\]

\begin{prop}\label{t:altbd}
Let $n \geqs 6$ be an even integer. Then $\lceil\log_2 n \rceil - 1 \leqs \gamma_u(A_n) \leqs 2\lceil\log_2 n \rceil$.
\end{prop}

\begin{proof}
Write $G=A_n$ and $n=2m$. If $n=6$ then by a direct computation in \textsc{Magma} (see Section~\ref{sss:comp_search}) we can prove that $\gamma_u(G) = 4$; indeed, $\Gamma(G)$ has a total dominating set comprising four conjugate 5-cycles. 

Now assume that $n \geqs 8$. First we establish the upper bound. Set $l = (m-1,2)$ and fix $s \in G$ with shape $[k,n-k]$, where $k = m-l$. By the proof of Theorem~\ref{t:AltMax}, $\M(G,s) = \{ H \}$ with $H = (S_k \times S_{n-k}) \cap G$.  By combining Corollary~\ref{c:CriterionBase} and Lemma~\ref{thm:Halasi}, we get  
\[ 
\gamma_u(G) \leqs b(G,G/H) \leqs \left\lceil \log_{\left\lceil \frac{2n}{n-2l} \right\rceil} n \right\rceil \cdot \left\lceil \frac{n+2l}{n-2l} \right\rceil \leqs 2\lceil \log_2 n \rceil
\] 
as required.

To prove the lower bound, fix an element $1 \ne s \in G$. Since $n$ is even, $s$ is not an $n$-cycle. Therefore, $s$ is contained in the stabiliser $H$ of a proper subset of $\{1,\dots,n\}$, and we have $b(G,G/H) \geqs \lceil \log_2 n \rceil - 1$ by Lemma~\ref{thm:Halasi}. By applying Corollary~\ref{c:LowerBound}, we conclude that $\gamma_u(G) \geqs \lceil \log_2 n \rceil - 1$.
\end{proof}

\subsubsection{Odd degree}
To complete the proof of Theorem~\ref{t:mainn}, we may assume $G = A_n$, where $n \geqs 9$ is a composite odd integer. 

\begin{prop}\label{t:altodd}
Let $n \geqs 5$ be an odd integer. Then $\gamma_u(A_n) \leqs 77\log_2 n$.
\end{prop}

In order to prove Proposition~\ref{t:altodd}, we need some preliminary results. We start with the following technical lemma. Note that for the remainder of this section, we adopt the standard convention that $\binom{a}{b} = 0$ if $b > a$. 

\begin{lem}\label{l:BinomLemma}
Let $l$ and $m$ be integers such that $l \geqs 2$ and $0 \leqs m \leqs 4l$. Then 
\[
f(l,m) \coloneqq \sum_{j=0}^{\min\{l, \lfloor{m/2}\rfloor\}} \binom{l}{j}\binom{4l}{2m-4j} \geqs \binom{4l}{m}.
\]
\end{lem}

\begin{proof}
First assume $0 \leqs m \leqs 2l-2$, so $\lfloor{m/4}\rfloor \leqs \min\{l, \lfloor{m/2}\rfloor\}$ and we can consider the term 
\[ 
\binom{l}{\lfloor{m/4}\rfloor}\binom{4l}{2m-4\lfloor{m/4}\rfloor} 
\] 
in $f(l,m)$. If $m \leqs 2l-3$, then $m-3 \leqs 4\lfloor{m/4}\rfloor \leqs m$, so $m \leqs 2m-4\lfloor{m/4}\rfloor \leqs m+3 \leqs 2l$. Similarly, if $m = 2l-2$ then $2l-4 \leqs 4\lfloor{(2l-2)/4}\rfloor \leqs 2l-2$ and $m \leqs 2m-4\lfloor{m/4}\rfloor \leqs 2l$ in this case also. Therefore, 
\[
f(l,m) \geqs \binom{l}{\lfloor{m/4}\rfloor}\binom{4l}{2m-4\lfloor{m/4}\rfloor} \geqs \binom{4l}{2m-4\lfloor{m/4}\rfloor} \geqs \binom{4l}{m}. 
\]
A very similar argument applies if $2l+2 \leqs m \leqs 4l$, working with the term corresponding to $j=\lceil{m/4}\rceil$. We omit the details.

Finally, let us assume $m \in \{2l-1, 2l, 2l+1\}$. We will provide the details for $m=2l-1$; the other cases are very similar. The result is easily verified if $l=2$, so assume that $l \geqs 3$. Observe that $2l-2 \leqs 2m - 4\lfloor{m/4}\rfloor \leqs 2l$, so 
\[
\binom{4l}{2m-4\lfloor{m/4}\rfloor} \geqs \binom{4l}{2l-2} = \frac{2l-1}{2l+2} \binom{4l}{2l-1}  \geqs \frac{1}{2} \binom{4l}{m}.
\]
Additionally, since $m = 2l-1 \geqs 5$, we have
\[
\binom{l}{\lfloor{m/4}\rfloor} \geqs l \geqs 3
\]
and thus
\[
f(l,m) \geqs \binom{l}{\lfloor{m/4}\rfloor}\binom{4l}{2m-4\lfloor{m/4}\rfloor} \geqs 3 \cdot \frac{1}{2} \binom{4l}{m} \geqs \binom{4l}{m}.
\]
This completes the proof.
\end{proof}

The next result on fixed point ratios is a key ingredient in the proof of Proposition~\ref{t:altodd}. For $g \in G$, we write ${\rm supp}(g)$ to denote the support of $g$.

\begin{lem}\label{l:AltOddFpr}
Suppose $G = S_n$, where $n$ is odd and let $s \in G$ be an element with shape as in \eqref{e:s}. Fix an element $x \in G$ with shape $[d^r,1^{n-dr}]$ for some $d \geqs 2$ and $r \geqs 1$. Set $t = |{\rm supp}(g)| = dr$. If $t \geqs 100$, then
\[
\fpr(x,G/H) < 0.98^t
\]
for all $H \in \M(G,s)$.
\end{lem}

\begin{proof}
As noted in the proof of Theorem~\ref{t:AltMax}, $\M(G,s)$ comprises three intransitive subgroups. Write $H = S_k \times S_{n-k} \in \M(G,s)$, where $k < \frac{n}{2}$, and identify $G/H$ with the set $\O$ of $k$-element subsets of $\{1, \ldots, n\}$. Since a set $A \in \O$ is fixed (setwise) by $x$ if and only if each cycle of $x$ is contained in or disjoint from $A$, it follows that  
\[
\fpr(x,G/H) = \sum_{i=0}^{r} \binom{r}{i} f(i)
\]
where
\[
f(i) = \frac{\binom{n-t}{k-di}}{\binom{n}{k}} = \frac{k\cdots(k-di+1)(n-k)\cdots(n-t-k+di+1)}{n \cdots (n-t+1)}.
\]
Note that
\[
f(i) \leqs \min\left\{\left(\frac{k}{n}\right)^{di}, \left(1-\frac{k}{n}\right)^{t-di}\right\}.
\]

Since $n \geqs t \geqs 100$, from the shape of $s$ in \eqref{e:s} it follows that $31/99 \leqs k/n \leqs 35/99$. Therefore, if $di \geqs 0.265t$ then 
\[
f(i) \leqs \left( \frac{k}{n} \right)^{0.265t} \leqs \left( \left(\frac{35}{99}\right)^{0.265} \right)^t < 0.76^t,
\]
otherwise
\[
f(i) \leqs \left( \frac{n-k}{n} \right)^{0.735t} \leqs \left( \left(\frac{68}{99}\right)^{0.735} \right)^t < 0.76^t.
\]
It follows that
\[
\fpr(x,G/H) < \sum_{i=0}^{r} \binom{r}{i} 0.76^t
            = 2^r \cdot 0.76^t 
            = (2^{1/d} \cdot 0.76)^t
\]
and thus 
\begin{equation}\label{e:dg3}
\fpr(x,G/H) < (2^{1/3} \cdot 0.76)^t < 0.9576^t
\end{equation}
if $d \geqs 3$.

The case $d=2$ requires special attention. Here $x$ has cycle shape $[2^r,1^{n-2r}]$ with $r \geqs 50$. Set $l=\lfloor r/4 \rfloor$ and fix elements $y$ and $z$ in $G$ of shape $[2^{4l},1^{n-8l}]$ and $[4^l, 1^{n-4l}]$, respectively. Without loss of generality, we may assume that 
\begin{align*}
x & = (1,2)(3,4) \cdots (2r-1,2r) \\
y & = (1,2)(3,4) \cdots (8l-1,8l) \\
z & = (1,2,3,4)  \cdots (4l-3,4l-2,4l-1,4l).
\end{align*}
We claim that 
\[
\fpr(x,\O) \leqs \fpr(y,\O) \leqs \fpr(z,\O).
\]

Let ${\rm Fix}(g,\O)$ be the set of fixed points of $g$ on $\O$. Since a set $A \in \O$ is fixed by $g$ if and only if each cycle of $g$ is contained in or disjoint from $A$, it follows that ${\rm Fix}(x,\O) \subseteq {\rm Fix}(y,\O)$ and thus $\fpr(x,\O) \leqs \fpr(y,\O)$. 

For $0 \leqs m \leqs k/2$ define
\begin{align*}
Y_m &= \{ A \cap {\rm supp}(y) \,:\,  A \in {\rm Fix}(y,\O), |A \cap {\rm supp}(y)| = 2m \} \\
Z_m &= \{ A \cap {\rm supp}(y) \,:\, A \in {\rm Fix}(z,\O), |A \cap {\rm supp}(y)| = 2m \}.
\end{align*}

By counting the sets $A \in {\rm Fix}(y,\O)$ according to the size of $A \cap {\rm supp}(y)$, we see that 
\[
|{\rm Fix}(y,\O)| = \sum_{m=0}^{\min\{4l,\lfloor{k/2}\rfloor\}} \binom{n-8l}{k-2m} |Y_m|,
\] 
and similarly
\[
|{\rm Fix}(z,\O)| = \sum_{m=0}^{\min\{4l,\lfloor{k/2}\rfloor\}} \binom{n-8l}{k-2m} |Z_m|.
\] 
In particular, $\fpr(y,\O) \leqs \fpr(z,\O)$ if $|Y_m| \leqs |Z_m|$ for all $0 \leqs m \leqs k/2$.

Since a subset of ${\rm supp}(y)$ is fixed by $y$ if and only if it is a union of cycles of $y$, we have
\[
|Y_m| = \binom{4l}{m}.
\]
Similarly, a subset $A$ of ${\rm supp}(y)$ is fixed by $z$ if and only if it is a union of cycles of $z$, that is, $A = A_1 \cup A_2$ where $A_1$ is the support of a collection of $4$-cycles of $z$, and $A_2$ is a subset of ${\rm supp}(y) \setminus {\rm supp}(z) = \{4l+1, \ldots, 8l\}$ of the appropriate size. By considering the possible $4$-cycles corresponding to $A_1$, we deduce that 
\[
|Z_m| = \sum_{j=0}^{\min\{l,\lfloor{m/2}\rfloor\}}\binom{l}{j}\binom{4l}{2m-4j}.
\]
By Lemma~\ref{l:BinomLemma}, we have $|Y_m| \leqs |Z_m|$ for all $0 \leqs m \leqs k/2$, hence $\fpr(y,\O) \leqs \fpr(z,\O)$ as claimed. 

In view of the above bound for $d=4$ (see \eqref{e:dg3}), it follows that 
\[
\fpr(x,G/H) \leqs \fpr(z,G/H) < 0.9576^{4l}
\]
where $4l = 4\lfloor{r/4}\rfloor \geqs r-3 \geqs 47t/100$, and thus
\[
\fpr(x,G/H) < 0.9576^{47t/100} \leqs 0.98^t
\] 
as required.
\end{proof}

Finally, we are now in a position to prove Proposition~\ref{t:altodd}, which completes the proof of Theorem~\ref{t:mainn}.

\begin{proof}[Proof of Proposition~\ref{t:altodd}]
First assume $n \leqs 19$. If $n$ is a prime then $\gamma_u(G) \leqs 4$ by 
Proposition~\ref{t:main_alt} and Remark~\ref{r:all}. For $n \in \{9,15\}$, a straightforward computation in \textsc{Magma} shows that $G$ has a total dominating set consisting of $6$ conjugate $n$-cycles (see Section~\ref{sss:comp_prob}). For the remainder, we may assume $n \geqs 21$.

We apply the probabilistic method in Lemma~\ref{l:ProbMethod}. Let $s \in G$ be an element with shape as in \eqref{e:s} and suppose $x \in G$ has prime order. As in the proof of Lemma~\ref{l:ProbMethod}, write
\[
P(x,s) =  \frac{|\{z \in s^G \,:\, G \ne  \< x,z \>\}|}{|s^G|}
\]
and recall that
\[
P(x,s) \leqs \sum_{H \in \M(G,s)}^{} \fpr(x,G/H).
\]
For $t \in \{3, \ldots, n\}$, let $N_t$ be the number of elements in $G$ with support of size $t$. By \eqref{e:pxs},
\[
Q(G,s,c) \leqs \sum_{t=3}^{n} N_t \cdot \xi(t,s)^c,
\]
where 
\[
\xi(t,s) = \max\{P(x,s) \,:\, \mbox{$x \in G$ has prime order and $|{\rm supp}(x)| = t$}\}.
\]
Note that $N_t \leqs n^t$. 

Set $\alpha = 0.991$ and write $|{\rm supp}(x)|=t$. By the proof of \cite[Proposition~7.1]{GK}, using the fact that $n \geqs 21$, we see that 
$P(x,s) \leqs 0.9$ if $t \in \{3,4\}$, and $P(x,s) \leqs 0.36$ if $t \geqs 5$.  Therefore, $\xi(t,s) \leqs \alpha^{t+1}$ if $t < 100$. Similarly, if $t \geqs 100$ then Lemma~\ref{l:AltOddFpr} implies that 
\[
\xi(t,s) \leqs 3 \cdot 0.98^t = 3 \cdot 0.98^{t/2-1/2} (0.98^{1/2})^{t+1} \leqs \alpha^{t+1}.
\]
Therefore, $\xi(t,s) \leqs \alpha^{t+1}$ for all $t$. As a result, if $c = \log_{1/\alpha}{n} \leqs 77 \log_2{n}$, then
\[
Q(G,s,c) \leqs \sum_{t=3}^{n} n^t (\alpha^{t+1})^c = \sum_{t=3}^{n} n^t \left( \frac{1}{n} \right)^{t+1} = \sum_{t=3}^{n} \frac{1}{n} = \frac{n-2}{n} < 1
\]
and we conclude that $\gamma_u(G) \leqs c \leqs 77 \log_2{n}$.
\end{proof}


\section{Sporadic simple groups}\label{s:spor}

\subsection{Maximal overgroups}\label{ss:spor_max}

In this section we determine the exact value of $\mu(G)$ for all sporadic simple groups $G$.

\begin{thm}\label{t:SporadicMax}
For each sporadic simple group $G$, the value of $\mu(G)$ is recorded in Table~\ref{tab:Sporadic}. In particular, $\mu(G) \leqs 3$, with equality if and only if $G \in \{{\rm M}_{12}, {\rm J}_2, {\rm McL}, {\rm Suz} \}$.
\end{thm}

\begin{proof}
If $G = \mathbb{B}$ or $\mathbb{M}$, then $\mu(G) = 1$ (see \cite[Table~IV]{GK}); in the final two rows of Table~\ref{tab:Sporadic} we present an element $s \in G$ (using \textsc{Atlas} \cite{ATLAS} notation to identify the conjugacy class of $s$) such that $|\M(G,s)|=1$. In each of the remaining cases, we can use \textsf{GAP} to determine $|\M(G,s)|$ for each conjugacy class representative $s \in G$ (see  Section~\ref{sss:comp_max}). In this way, we compute $\mu(G)$ and we identify an element $s \in G$ with $|\M(G,s)|=\mu(G)$. This information is presented in Table~\ref{tab:Sporadic}. (For $G \not\in \{{\rm Co}_1, {\rm Fi}_{22}, {\rm Fi}_{23}\}$, the value of $\mu(G)$ given in Table~\ref{tab:Sporadic} was known to be an upper bound; see \cite[Table IV]{GK} and \cite[Table 7]{BGK}.)
 \end{proof}

\begin{table}
\begin{center}
\[\begin{array}{lccclcc} 
\hline
G              & \mu(G) & \gamma_u(G)       & s            & \M(G,g)                                                & b & c \\ 
\hline
{\rm M}_{11}   & 1      & 4                 & \texttt{11A} & {\rm L}_2(11)                                          & 4 &   \\
{\rm M}_{12}   & 3      & 4 & \texttt{10A} & A_6.2^2, \: A_6.2^2, \: 2 \times S_5                   &   & 6 \\ 
{\rm M}_{22}   & 1      & 3                 & \texttt{11A} & {\rm L}_2(11)                                          & 3 &   \\
{\rm M}_{23}   & 1      & 2                 & \texttt{23A} & 23{:}11                                                & 2 &   \\
{\rm M}_{24}   & 2      & \mbox{$3$ or $4$} & \texttt{21A} & {\rm L}_3(4){:}S_3, \: 2^6{:}({\rm L}_3(2) \times S_3) &   & 4 \\

{\rm J}_1      & 1      & 2                 & \texttt{15A} & D_6 \times D_{10}                                      & 2 &   \\  
{\rm J}_2      & 3      & \mbox{$3$ or $4$} & \texttt{10C} & 2^{1+4}{:}A_5, \: A_5 \times D_{10}, \: 5^2{:}D_{12}   &   & 4 \\
{\rm J}_3      & 2      & \mbox{$2$ or $3$} & \texttt{19A} & {\rm L}_{2}(19), \: {\rm L}_2(19)                      &   & 3 \\
{\rm J}_4      & 1      & 2                 & \texttt{43A} & 43{:}14                                                & 2 &   \\

{\rm HS}       & 2      & \mbox{$3$ or $4$} & \texttt{15A} & S_8, \: 5{:}4 \times A_5                               &   & 4 \\
{\rm Suz}      & 3      & 3                 & \texttt{14A} & {\rm J}_2{:}2, \: {\rm J}_2{:}2, \: (A_4 \times {\rm L}_{3}(4)){:}2   &   & 3 \\
{\rm McL}      & 3      & 3                 & \texttt{15A} & 3^{1+4}{:}2.S_5, \: 2.A_8, \: 5^{1+2}{:}3{:}8          &   & 3 \\
{\rm Ru}       & 1      & 2                 & \texttt{29A} & {\rm L}_{2}(29)                                        & 2 &   \\ 
{\rm He}       & 1      & \mbox{$2$ or $3$} & \texttt{17A} & {\rm Sp}_{4}(4){:}2                                    & 4 &   \\
               &        &                   & \texttt{21A} & 3.S_7, \: 7^{1+2}{:}(3 \times S_3),                    &   & 3 \\  
               &        &                   &              & 7{:}3 \times {\rm L}_3(2), \: 7{:}3 \times {\rm L}_3(2) &  &   \\
{\rm Ly}       & 1      & 2                 & \texttt{28A} & 2.A_{11}                                               & 2 &   \\
{\rm O'N}      & 2      & 2                 & \texttt{31A} & {\rm L}_{2}(31), \: {\rm L}_{2}(31)                    &   & 2 \\
{\rm Co}_1     & 1      & \mbox{$2$ or $3$} & \texttt{26A} & (A_4 \times G_2(4)){:}2                                & 3 &   \\
{\rm Co}_2     & 1      & 3                 & \texttt{23A} & {\rm M}_{23}                                           & 3 &   \\
{\rm Co}_3     & 1      & 3                 & \texttt{23A} & {\rm M}_{23}                                           & 3 &   \\
{\rm Fi}_{22}  & 1      & \mbox{$3$ or $4$} & \texttt{22A} & 2.{\rm U}_6(2)                                         & 5 &   \\
               &        &                   & \texttt{16A} & 2^{5+8}{:}(S_3 \times A_6), \: 2.2^{1+8}{:}({\rm U}_{4}(2){:}2)  &   & 4 \\
               &        &                   &              & 2^{10}{:}{\rm M}_{22}, \:{}^2F_{4}(2)' \: \mbox{(4 times)}       &   &   \\ 
{\rm Fi}_{23}  & 1      & 2                 & \texttt{35A} & S_{12}                                                 & 2 &   \\
{\rm Fi}_{24}' & 1      & 2                 & \texttt{29A} & 29{:}14                                                & 2 &   \\ 
{\rm HN}       & 1      & \mbox{$2$ or $3$} & \texttt{22A} & 2.{\rm HS}.2                                           & 3 &   \\
{\rm Th}       & 2      & 2                 & \texttt{19A} & {\rm U}_3(8){:}6, \: {\rm L}_2(19){:}2                 &   & 2 \\
\mathbb{B}     & 1      & 2                 & \texttt{47A} & 47{:}23                                                & 2 &   \\
\mathbb{M}     & 1      & 2                 & \texttt{59A} & {\rm L}_{2}(59)                                        & 2 &   \\ 
\hline
\end{array}\]
\caption{Sporadic simple groups} \label{tab:Sporadic}
\end{center}
\end{table}

\subsection{Uniform domination number}\label{ss:spor_udn}

Our main result on the uniform domination number of sporadic groups is the following. Note that this immediately implies Theorem~\ref{t:main3}.

\begin{thm}\label{t:SporadicUDN}
Let $G$ be a sporadic simple group. Then 
\[
d-\e \leqs \gamma_u(G) \leqs d,
\]
where $d$ is defined as follows:
\[
{\renewcommand{\arraystretch}{1.1}
\begin{array}{lccccccccccccc}
\hline
G & {\rm M}_{11} & {\rm M}_{12} & {\rm M}_{22} & {\rm M}_{23} & {\rm M}_{24} & {\rm J}_{1} & {\rm J}_{2}   & {\rm J}_{3}   & {\rm J}_{4}    & {\rm HS} & {\rm Suz} & {\rm McL}  & {\rm Ru}   \\ 
d & 4            & 4          & 3            & 2            & 4^*          & 2           & 4^*           & 3^*           & 2              & 4^*      & 3         & 3          & 2          \\
\hline
  &              &              &              &              &              &             &               &               &                &          &           &            &            \\          
\hline
  & {\rm He}     & {\rm Ly}     & {\rm O'N}    & {\rm Co}_1   & {\rm Co}_2   & {\rm Co}_3  & {\rm Fi}_{22} & {\rm Fi}_{23} & {\rm Fi}_{24}' & {\rm HN} & {\rm Th}  & \mathbb{B} & \mathbb{M} \\ 
  & 3^*          & 2            & 2            & 3^*          & 3            & 3           & 4^*           & 2             & 2              & 3^*      & 2         & 2          & 2          \\
\hline
\end{array}}
\]
Here an asterisk indicates that $\e=1$; otherwise $\e=0$ and $\gamma_u(G) = d$. In particular, $\gamma_u(G) \leqs 4$, with equality if $G = {\rm M}_{11}$ or ${\rm M}_{12}$. 
\end{thm}

\begin{proof}
If $\mu(G) = 1$ then we choose an element $s \in G$ such that $\M(G,s) = \{H\}$ and $b(G,G/H)$ is minimal (note that the base size of every almost simple primitive group with sporadic socle has been computed; see \cite{BOW, NNOW}). The element $s$, subgroup $H$ and base size $b = b(G,G/H)$ are recorded in Table~\ref{tab:Sporadic}, and we note that $\gamma_u(G) \leqs b$ by Corollary~\ref{c:CriterionBase}. In particular, we conclude that $\gamma_u(G)=2$ if $G=  \mathbb{B}$ or $\mathbb{M}$. 

For each sporadic group $G \not\in \{ \mathbb{B},\mathbb{M} \}$ and each class representative $s \in G$ we use \textsf{GAP} to determine the minimal  $c$ such that $\widehat{Q}(G,s,c) < 1$ (see \eqref{e:qsc}), following the method described in Section~\ref{sss:comp_prob}. By Lemma~\ref{l:ProbMethod}, we have $\gamma_u(G) \leqs c$. For the groups with $\mu(G)=1$, we almost always find that $b \leqs c$; the exceptions are the cases $G \in \{ {\rm He}, {\rm Fi}_{22} \}$, where it is better to apply Lemma~\ref{l:ProbMethod} with an element $s \in G$ for which $|\M(G,s)| > 1$. For these two groups, and also for those with $\mu(G)>1$, we record the minimal value of $c$ in Table~\ref{tab:Sporadic}, together with an element $s \in G$ such that $\widehat{Q}(G,s,c) < 1$. For example, if $G = {\rm He}$ and $s \in \texttt{17A}$ then Table~\ref{tab:Sporadic} indicates that $\M(G,s) = \{H\}$ with $H = {\rm Sp}_{4}(4).2$ and $b(G,G/H)=4$. However, if we choose $s \in \texttt{21A}$ then $|\M(G,s)| = 4$ and $\widehat{Q}(G,s,3)<1$, so $\gamma_u(G) \in \{2,3\}$.

To derive a lower bound on $\gamma_u(G)$, we proceed as in Remark~\ref{r:LowerBound} (see Section~\ref{sss:comp_max}). For example, if $G = {\rm M}_{11}$ then every element of $G$ is contained in a subgroup $H$ isomorphic to ${\rm L}_2(11)$ or ${\rm M}_{10}$; in both cases $b(G,G/H) = 4$, so Corollary~\ref{c:LowerBound} implies that $\gamma_u(G) \geqs 4$. With the exception of $G={\rm M}_{12}$, this explains how we obtain the results on $\gamma_u(G)$ presented in Table~\ref{tab:Sporadic}. 

The case $G = {\rm M}_{12}$ requires special attention. Here the above approach only gives $3 \leqs \gamma_u(G) \leqs 6$, but by carrying out a random search in \textsc{Magma} (see Section~\ref{sss:comp_search}) one can show that the class $\texttt{10A}$ contains a total dominating set for $\Gamma(G)$ of size $4$ and thus 
$\gamma_u(G) \in \{3,4\}$. To rule out the existence of a uniform dominating set of size $3$, we first combine the base size results in \cite{BOW} with 
Corollary~\ref{c:LowerBound} to reduce the problem to the classes labelled \texttt{3B} and \texttt{6A}. In fact, since the square of an element in \texttt{6A} is in the class \texttt{3B}, we only need to consider \texttt{6A}. The required exhaustive search can now be carried out in {\sc Magma} and we refer the reader to \cite[Section~1.2.4]{BH_comp} for further details of this computation. We conclude that $\gamma_u(G)=4$.
\end{proof}


\section{Exceptional groups of Lie type}\label{s:ex}

Let us now assume $G$ is a simple group of Lie type over $\mathbb{F}_q$, where $q=p^a$ and $p$ is prime. In this section we prove Theorems~\ref{t:main4} and~\ref{t:main5} for exceptional groups of Lie type; the classical groups will be handled in Section~\ref{s:cla}.

\subsection{Maximal overgroups}\label{ss:ex_max}

\begin{thm}\label{t:ExceptionalMax}
Let $G$ be a finite simple exceptional group of Lie type over $\mathbb{F}_q$, where $q=p^a$ and $p$ is prime. Then either $\mu(G)=1$ or one of the following holds:
\begin{itemize}\addtolength{\itemsep}{0.2\baselineskip}
\item[{\rm (i)}]  either $G = F_4(q)$ with $p=2$, or $G = G_2(q)$ with $q > p = 3$, and $\mu(G) \leqs 2$;
\item[{\rm (ii)}] $(G,\mu(G)) \in \{ ({}^2F_4(2)',2), (G_2(3),3) \}$.
\end{itemize}
In particular, $\mu(G) \leqs 3$ with equality if and only if $G = G_2(3)$.
\end{thm}

\begin{proof}
This is essentially an immediate corollary of the work of Weigel in \cite{Wei}. Set
\[
\mathcal{E} = \{E_7(2), E_7(3),  {}^2E_6(2),  {}^2E_6(3), F_4(3), F_4(2), {}^2F_4(2)',  G_2(3), G_2(4)\}.
\]
For $G \not\in \mathcal{E}$, Weigel identifies an element $s \in G$ with $|\mathcal{M}(G,s)|=1$, with the possible exception of the groups $F_4(2^a)$ and $G_2(3^a)$ (with $a \geqs 2$), where he finds an element $s$ with $|\mathcal{M}(G,s)|=2$. (In every case, $s$ is a generator of a maximal torus of $G$.) 

To complete the proof, we just need to handle the groups in $\mathcal{E}$. If $G$ is one of the first five groups in $\mathcal{E}$, then \cite[Proposition~6.2]{GK} implies that $\mu(G)=1$. If $G$ is one of the remaining four groups, then we can use \textsf{GAP} to compute $\mu(G)$ and find an element $s \in G$ with $|\M(G,s)|=\mu(G)$ (see Section~\ref{sss:comp_max}): 
\[
{\renewcommand{\arraystretch}{1.1}
\begin{array}{lcccc}
\hline
G      & F_4(2)       & ^{2}F_4(2)'  & G_2(3)       & G_2(4)       \\ 
\mu(G) & 2            & 2            & 3            & 1            \\
s      & \mathtt{17A} & \mathtt{16A} & \mathtt{13A} & \mathtt{21A} \\
\hline
\end{array}}
\]
This completes the proof.
\end{proof}

\subsection{Uniform domination number}\label{ss:ex_udn}

Our main result on the uniform domination number of exceptional groups is the following theorem, which proves Theorem~\ref{t:main4}(ii). Note that this also gives infinitely many more examples with $\gamma_u(G)=2$ (see Theorem~\ref{t:main1}).

\begin{thm}\label{t:ExceptionalUDN}
If $G$ is a finite simple exceptional group of Lie type, then $\gamma_u(G) \leqs 6$. Moreover, if $G \in \{{}^2B_2(q), {}^2G_2(q), E_8(q)\}$ then $\gamma_u(G) = 2$.
\end{thm}

\begin{proof}
First assume that $\mu(G) = 1$, and let $s \in G$ such that $\M(G,s) = \{H\}$. By combining Corollary~\ref{c:CriterionBase} and the main theorem of \cite{BLS}, we deduce that $\gamma_u(G) \leqs b(G,G/H) \leqs 6$. 

Next suppose that $G \in \{ {}^2B_2(q), {}^2G_2(q) \}$. By \cite{Wei}, there is an element $s \in G$ such that $\M(G,s) = \{H\}$ where $H = N_G(\la s \ra)$. (More precisely, $|s| = q+\sqrt{2q}+1$ if $G = {}^2B_2(q)$, and $|s|=q+\sqrt{3q}+1$ if ${}^2G_2(q)$.) In both cases, by applying \cite[Lemmas~4.37 and~4.39]{BLS}, we get $b(G,G/H) = 2$ and thus $\gamma_u(G) = 2$ as claimed. 

Now assume $G = E_8(q)$. Here we apply the probabilistic method from Lemma~\ref{l:ProbMethod} and we adopt the notation therein. Fix an element $s \in G$ of order $q^8+q^7-q^5-q^4-q^3+q+1$. By \cite[Section 4(j)]{Wei}, $\M(G,s) = \{H\}$ where $H = N_G(\la s \ra)$. Moreover, 
\[
|x^G \cap H|< |H| = 30(q^8+q^7-q^5-q^4-q^3+q+1)< q^{14}
\]
(see \cite[Theorem~5.2]{LSS}, for example) and $|x^G|>q^{58}$ for every element $x \in G$ of prime order (indeed, $|x^G|$ is minimal when $q$ is even and $x$ is a long root element). Therefore, by \cite[Proposition~2.3]{BLS},
\[
Q(G,s,2) \leqs \sum_{i=1}^{k}|x_i^G|\cdot \fpr(x_i,G/H)^2  <  q^{58}(q^{-44})^2 = q^{-30}<1
\]
and we conclude that $\gamma_u(G)=2$.

To complete the proof, it remains to handle the cases with $\mu(G) > 1$. In two cases we employ computational methods. Indeed, for 
\[ 
(G,s)\in \{ ({}^2F_4(2)', \mathtt{16A}), \, (G_2(3), \mathtt{13A})\}
\] 
we can use \textsf{GAP} to verify the bound $\widehat{Q}(G,s,5) < 1$ (see Section~\ref{sss:comp_prob}). By Lemma~\ref{l:ProbMethod}, this gives $\gamma_u(G) \leqs 5$.

Next assume $G = F_4(q)$ with $q = 2^a$. There is an element $s \in G$ with $\M(G,s) = \{ H, K \}$, where $H \cong K \cong {}^3D_4(q).3$ if $a>1$ (see \cite[Section~4(f)]{Wei}) and $H \cong K \cong {\rm Sp}_8(2)$ if $a=1$ (see \cite[Proposition~6.2]{GK}). Since $H$ and $K$ are ${\rm Aut}(G)$-conjugate, it follows that 
\begin{equation}\label{e:20}
\widehat{Q}(G,s,c) \leqs 2^c\sum_{i=1}^k|x_i^G|\cdot \fpr(x_i,G/H)^c
\end{equation}
in terms of the notation of Lemma~\ref{l:ProbMethod}.

Suppose $a=1$. The \textsf{GAP} Character Table Library contains the character tables of $G$ and $H$, so as described in Section~\ref{sss:comp_prob}, we can compute $\fpr(x,G/H)$ for all prime order elements $x \in G$. In this way, we deduce that $\what{Q}(G,s,5) < 1$ and thus $\gamma_u(G) \leqs 5$.  

Now assume $a>1$. By inspecting the proof of \cite[Lemma~4.26]{BLS}, we deduce that 
\[
\widehat{Q}(G,s,6) < 64\sum_{i=1}^{5}a_ib_i^6,
\]
where $a_i,b_i$ are defined as follows:
\[
{\renewcommand{\arraystretch}{1.1}
\begin{array}{lccccc} 
\hline
i   & 1      & 2       & 3       & 4       & 5        \\ 
a_i & q^{52} & 2q^{31} & 2q^{16} & 3q^{22} & q^{48}   \\
b_i & q^{-9} & q^{-6}  & 2q^{-5} & 2q^{-6} & 8q^{-12} \\ 
\hline
\end{array}}
\]
It is easy to check that this yields $\widehat{Q}(G,s,6)<1$ for all $q \geqs 16$. For $q \in \{4,8\}$, one checks that the value of $q^{-9}$ for $b_1$ can be replaced by $q^{-11}$ and this minor modification yields $\widehat{Q}(G,s,6)<1$. Therefore, $\gamma_u(G) \leqs 6$ as required.

A similar argument applies when $G = G_2(q)$ with $q = 3^a$ and $a \geqs 2$. As explained in \cite[Section 4(d)]{Wei}, there is an element $s \in G$ such that $\M(G,s) = \{H,K\}$ and $H \cong K \cong {\rm SU}_{3}(q).2$. In particular, \eqref{e:20} holds and one can check that 
\[
\widehat{Q}(G,s,6) < 64\sum_{i=1}^{4}a_ib_i^6 < 1,
\]
where the $a_i,b_i$ are given in the proof of \cite[Lemma~4.31]{BLS}, hence $\gamma_u(G) \leqs 6$.
\end{proof}


\section{Classical groups}\label{s:cla}

In this final section, we study the parameters $\mu(G)$ and $\gamma_u(G)$ when $G$ is a finite simple classical group. In particular, we complete the proofs of Theorems~\ref{t:main4} and~\ref{t:main5}. 

Throughout this section, we will write $r$ for the untwisted Lie rank of $G$ (that is, $r$ is the rank of the ambient simple algebraic group). Due to the existence of isomorphisms between certain low rank classical groups (see \cite[Proposition 2.9.1]{KL}, for example), we may (and will) assume that $G$ is one of the following:
\[
{\rm L}_{r+1}(q), \, r \geqs 1; \:\: {\rm U}_{r+1}(q), \, r \geqs 2; \:\:  {\rm PSp}_{2r}(q)', \, r \geqs 2; \:\:  {\rm P\Omega}^{\pm}_{2r}(q), \, r \geqs 4; \:\:  {\rm \Omega}_{2r+1}(q), \, r \geqs 3.
\]
In addition, we assume $q$ is odd if $G = \O_{2r+1}(q)$.

\subsection{Maximal overgroups}\label{ss:cla_max}

The main result of this section is the following theorem, which completes the proof of Theorem~\ref{t:main5}.

\begin{thm}\label{t:ClassicalMax}
If $G$ is a finite simple classical group, then either $\mu(G) \leqs 3$ or $(G,\mu(G))$ is one of the following:
\[
{\renewcommand{\arraystretch}{1.1}
\begin{array}{lcccc}
\hline
G      & {\rm U}_{6}(2) & {\rm U}_4(3) & \O_8^+(2) & {\rm P\O}_8^+(3) \\
\mu(G) & 4              & 5            & 7         & 7                \\
\hline
\end{array}}
\]
\end{thm}

In order to prove Theorem~\ref{t:ClassicalMax}, we need to introduce some additional notation and terminology, which will also be useful later in Section \ref{ss:cla_udn}. 

Let $G$ be a finite simple classical group over $\F$ with natural module $V$ of dimension $n$. We will write $k \oplus (n-k)$ to denote a decomposition $V = U \oplus W$, where $U$ and $W$ are totally singular subspaces of dimensions $k$ and $n-k$ (if $G = {\rm L}_n(q)$ then all subspaces are totally singular). In turn, $g = k \oplus (n-k)$ will denote a semisimple element $g \in G$ which preserves such a decomposition and acts irreducibly on both $U$ and $W$. Similarly, if $G \neq {\rm L}_n(q)$, then $k \perp (n-k)$ denotes an orthogonal decomposition $V = U \perp W$ where $U$ is a non-degenerate $k$-space, and we will write $g = k \perp (n-k)$ for an element in $G$ acting irreducibly on $U$ and $W$. For an orthogonal group, we extend this notation in the obvious way by writing $k^{\pm}$ to denote a non-degenerate $k$-space of type $\pm$ (with $k$ even). This is consistent with the notation used in \cite{BGK,GK}. Following \cite{KL}, we will sometimes refer to the \emph{type} of a maximal subgroup $H$ of $G$, which provides an approximate description of the group-theoretic structure of $H$.

Write $q=p^a$ for a prime $p$ and suppose $t$ is a prime divisor of $q^e-1$ for some  $e \geqs 2$. Recall that $t$ is a \emph{primitive prime divisor} (ppd for short) of $q^e-1$ if $t$ is not a divisor of $q^{i}-1$ for all $1 \leqs i <e$. A classical theorem of Zsigmondy \cite{Zsig} states that if $e \geqs 3$ then $q^e-1$ has a ppd unless $(q,e)=(2,6)$. Primitive prime divisors also exist when $e=2$, provided $q$ is not a Mersenne prime. Note that if $t$ is a ppd of $q^e-1$ then $t \equiv 1 \imod{e}$. 

The following lemma establishes a special case of Theorem \ref{t:ClassicalMax}.

\begin{lem}\label{lem:SymplecticMax}
Suppose $G={\rm Sp}_{2r}(q)$, where $r \geqs 6$ is even and $q$ is even. Let 
\[ 
s = (r-2k) \perp (r+2k) \in G,
\] 
where $k=(r/2-1,2)$. Then
\[
\M(G,s) = \{{\rm Sp}_{r-2k}(q) \times {\rm Sp}_{r+2k}(q), O^+_{2r}(q) \}
\]
and thus $\mu(G) \leqs 2$.
\end{lem}

\begin{proof}
Write $V = U \perp W$, where $U$ and $W$ are the proper non-degenerate subspaces preserved by $s$. First observe that the order of $s$ is divisible by a ppd $t$ of $q^{r+2k}-1$, so we are in a position to apply the main result of \cite{GPPS} to determine the subgroups in $\M(G,s)$. Following the notation of \cite{GPPS}, set $d=2r$ and $e = r+2k$. By the main theorem of \cite{GPPS}, every maximal overgroup of $s$ in $G$ is one of those listed in \cite[Examples 2.1--2.9]{GPPS}. We will consider each of these cases in turn.

Write $q=p^a$ and consider the classical groups arising in \cite[Example 2.1]{GPPS}. The element $s$ is not contained in any subfield subgroups since $t$ does not divide 
\[
|{\rm Sp}_{2r}(q_0)| = q_0^{r^2}\prod_{i=1}^{r} (q_0^{2i}-1)
\]
for $q_0=p^b$ and $b < a$. The orthogonal groups $O^{\pm}_{2r}(q)$ are the only other maximal subgroups that can arise in \cite[Example 2.1]{GPPS} (moreover, it is well known that every element in $G$ is contained in such a subgroup). First observe that if $s$ is contained in an orthogonal subgroup $H$ then the irreducibility of $s$ on $U$ and $W$ implies that both $U$ and $W$ are minus-type orthogonal spaces (with respect to the quadratic form corresponding to $H$), so $H = O_{2r}^{+}(q)$ is the only possibility. Moreover, we claim that $s$ is contained in exactly one such subgroup.  As noted in \cite[Table 3.5C]{KL}, $G$ contains a unique conjugacy class of subgroups $O^+_{2r}(q)$, so we just need to compute ${\rm fpr}(s,G/H)\cdot |G:H|$, which is the number of $G$-conjugates of $H$ containing $s$. Since conjugacy of semisimple elements of odd order in both $G$ and $H$ is determined by eigenvalues (in a suitable field extension of $\F$), it follows that $s^G \cap H = s^H$. Moreover, 
\[
|C_G(s)|=(q^{r/2+k}+1)(q^{r/2-k}+1)=|C_H(s)| 
\] 
and thus ${\rm fpr}(s,G/H)\cdot |G:H|=1$, as claimed.

The subgroups in \cite[Example 2.2]{GPPS} are reducible. Since $s$ acts irreducibly on both $U$ and $W$, it follows that the subspace stabiliser $G_U = {\rm Sp}_{r-2k}(q) \times {\rm Sp}_{r+2k}(q)$ is the only reducible maximal subgroup of $G$ containing $s$. No imprimitive subgroups arise from \cite[Example 2.3]{GPPS}. Since $q$ is even, the field extension subgroups in \cite[Example 2.4]{GPPS} have type ${\rm Sp}_{2r/l}(q^l)$ for a prime divisor $l$ of $r$. If $s$ is contained in such a subgroup, then $l$ must divide $r/2-k$ and $r/2+k$, but this is not possible because these numbers are coprime.

To complete the proof of the lemma, we need to show that there are no additional subgroups in $\M(G,s)$. To do this, we need to argue that none of the subgroups in \cite[Examples 2.5--2.9]{GPPS} can arise. 

Let us first observe that the conditions $r \geqs 6$ and $e=r+2k$ imply that if $t=e+1$ then $q=2$ and $r \in \{ 6, 8, 14, 16 \}$ (see \cite[Lemma 2.1(ii)]{GM}). Now \cite[Example 2.5]{GPPS} requires $t=e+1$ and $p$ odd, so no examples occur, and we can also rule out the cases in \cite[Examples 2.6(b,c) and 2.8]{GPPS} since $r \geqs 6$. In \cite[Example 2.6(a)]{GPPS} we have $t=e+1$, so $r \in \{ 6, 8, 14, 16 \}$ and $q=2$. Here $(G,H) = ({\rm Sp}_{2r}(2),S_{2r+2})$ and the embedding of $H$ in $G$ is afforded by the fully deleted permutation module for $H$ over $\mathbb{F}_2$. Since 
\[
|s| = {\rm lcm}\{2^{r/2-k}+1, 2^{r/2+k}+1\},
\]
we can easily rule out $r \in \{8,14,16\}$ by simply considering the orders of elements in $H$. Now assume $r=6$, so $|s|=33$ and $H$ has a unique class of elements of order $33$. We also note that $G$ has a unique conjugacy class of maximal subgroups isomorphic to $H$ (see \cite[Table 8.81]{BHR}). Finally, since $G$ has three classes of elements of order $33$ and type $2 \perp 10$, without any loss of generality we may assume that $\M(G,g)$ does not contain any subgroups isomorphic to $S_{14}$. 

Finally, the handful of cases with $r \geqs 6$ in \cite[Examples 2.7 and 2.9]{GPPS} are not compatible with our condition $e=r+2k$, unless $(r,q)=(6,2)$. In this case, the candidate maximal subgroup is almost simple with socle ${\rm L}_2(11)$, and this can be excluded since it does not contain an element of order $|s|=33$.  
\end{proof}

We are now ready to prove Theorem~\ref{t:ClassicalMax}.

\begin{proof}[Proof of Theorem~\ref{t:ClassicalMax}]
Let $G$ be a finite simple classical group over $\F$ with natural module $V$. First we appeal to the proof of the main theorem of \cite{GK}, which identifies an element $s \in G$ such that $\M(G,s)$ is small. 

If the rank $r$ of $G$ is large (for example, $r \geqs 11$ suffices), then this element is given in \cite[Table~II]{GK} and the remaining groups are covered in \cite[Section~5]{GK}, except for a short list of small groups which are handled in \cite[Proposition~6.3]{GK}. 
Moreover, additional information regarding the action of $s$ on $V$ is provided in \cite{GK}, which allows us to determine the precise subgroups in $\M(G,s)$. For example, if $G = {\rm P}\O^-_{2r}(q)$ where $r \geqs 7$ and $r \equiv 3 \imod{4}$, then following \cite[Table II]{GK} we choose 
\[
s = (r+1)^- \perp \left( \frac{r-1}{2} \oplus \frac{r-1}{2} \right).
\] 
By the proof of \cite[Proposition~4.1]{GK}, it follows that $\M(G,s) = \{ H, K_1, K_2 \}$ where $H$ has type $O^-_{r+1}(q) \times O^+_{r-1}(q)$ and both $K_1$ and $K_2$ are $P_{(r-1)/2}$ parabolic subgroups (that is, $K_1$ and $K_2$ are the stabilisers of totally singular subspaces of dimension $(r-1)/2$).
Similarly, \cite[Proposition 5.14]{BGK} implies that $\mu(G) \leqs 3$ if $G = {\rm P\O}_{2r}^{+}(q)$ and $r \geqs 4$ is even.

In this way, we deduce that $\mu(G) \leqs 3$, unless $G$ is one of the following:
\begin{itemize}\addtolength{\itemsep}{0.2\baselineskip}
\item[{\rm (a)}] ${\rm Sp}_{2r}(q)$ with $r \geqs 6$ even and $q$ even;
\item[{\rm (b)}] ${\rm P\O}^+_8(3), \, \O^+_8(2), \, \O_7(3), \, {\rm U}_6(2), \, {\rm Sp}_6(2), \, {\rm U}_4(3)$.
\end{itemize}

Case (a) was handled in Lemma~\ref{lem:SymplecticMax}. For the groups $G$ in (b), we can use \textsf{GAP} to determine $\mu(G)$ and to identify an element $s \in G$ with $|\M(G,s)|=\mu(G)$ (see Section~\ref{sss:comp_max}). We obtain the following results, in terms of the \textsc{Atlas} \cite{ATLAS} notation for conjugacy classes:
\[
{\renewcommand{\arraystretch}{1.1}
\begin{array}{lcccccc} 
\hline
G      & {\rm P\O}^+_8(3) & \O^+_8(2)     & \O_7(3)      & {\rm U}_6(2) & {\rm Sp}_6(2) & {\rm U}_4(3) \\
\mu(G) & 7                & 7             & 3            & 4            & 2             & 5            \\
s      & \texttt{14A}     & \texttt{15A}  & \texttt{14A} & \texttt{11A} & \texttt{15A}  & \texttt{9A} \\
\hline
\end{array}}
\] 
This proves the result.
\end{proof}

This completes the proof of Theorem~\ref{t:main5}.

\subsection{Uniform domination number}\label{ss:cla_udn}

Our main result is Theorem~\ref{t:ClassicalUDN}, which completes the proof of Theorem~\ref{t:main4}. In order to state this result, set 
\begin{align*}
\mathcal{A} & =\{ {\rm U}_{r+1}(q) \,:\, \text{$r \geqs 7$ odd} \} \cup \{ {\rm PSp}_{2r}(q) \,:\, \text{$r \geqs 3$ odd, $q$ odd} \} \cup \{ {\rm P\O}^+_{2r}(q) \,:\, \text{$r \geqs 5$ odd} \} \\
\mathcal{B} & = \{ {\rm Sp}_{2r}(q) \, :\, \text{$r \geqs 2$, $q$ even, $(r,q) \neq (2,2)$} \} \cup \{ \O_{2r+1}(q) \, :\, \text{$r \geqs 3$, $q$ odd} \}
\end{align*}

\begin{thm}\label{t:ClassicalUDN}
Let $G$ be a finite simple classical group of rank $r$. Then
\[
\gamma_u(G) \leqs 7r+56.
\]
More precisely, the following hold:
\begin{itemize}\addtolength{\itemsep}{0.2\baselineskip}
\item[{\rm (i)}]   If $G = {\rm L}_{2}(q)$, then $\gamma_u(G) \leqs 4$, with equality if and only if $q=9$. 
\item[{\rm (ii)}]  If $G \in \mathcal{A}$, then $\gamma_u(G) \leqs 15$.
\item[{\rm (iii)}] If $G \in \mathcal{B}$, then $r \leqs \gamma_u(G) \leqs 7r$.
\end{itemize}
\end{thm}

Note that the conclusion in part~(iii) of Theorem~\ref{t:ClassicalUDN} still holds for $G={\rm Sp}_{4}(2)'$, but it will be convenient to exclude this group from $\mathcal{B}$. Indeed, 
${\rm Sp}_4(2)' \cong A_6$ and the proof of Proposition~\ref{t:altbd} gives $\gamma_u(A_6) = 4$. 

We will prove Theorem~\ref{t:ClassicalUDN} in a sequence of propositions.

\subsubsection{Special cases}
We start by handling the special cases referred to in parts (i), (ii) and (iii). Note that part (iii) shows that the uniform domination number of the groups in $\mathcal{B}$ can be arbitrarily large. It also shows that the linear bound in Theorem \ref{t:ClassicalUDN} is essentially best possible (up to constants).

\begin{prop}\label{p:psl2}
If $q \geqs 4$, then $\gamma_u({\rm L}_2(q)) \leqs 4$ with equality if and only if $q=9$.
\end{prop}

\begin{proof}
For $q < 11$, the result can be verified computationally; see Section~\ref{sss:comp_prob}.
Now assume $q \geqs 11$. Set $d=(2,q-1)$ and fix an element $s \in G$ of order $(q+1)/d$. Then $\M(G,s) = \{ H \}$, where $H = N_G(\<g\>) \cong D_{2(q+1)/d}$ (see \cite[Section 5]{GK}). By combining Corollary~\ref{c:CriterionBase} and \cite[Lemma 4.5]{B07}, we conclude that $\gamma_u(G) \leqs b(G,G/H) \leqs 3$.
\end{proof}

In order to prove the bound in part (ii) of Theorem~\ref{t:ClassicalUDN}, we need the following recent result of Halasi, Liebeck and Mar\'{o}ti \cite[Theorem 3.3]{HLM} on the base sizes of subspace actions of classical groups.

\begin{prop}\label{p:ben}
Let $G$ be a finite simple classical group with natural module $V$ of dimension $n$. Let $H$ be the stabiliser of a $k$-dimensional subspace of $V$ with $k \leqs n/2$ and assume $H$ is a maximal subgroup of $G$. Then  
\[
b(G,G/H) \leqs \left\lfloor \frac{n}{k} \right\rfloor +11. 
\]
\end{prop}

\begin{prop}\label{p:ClassicalConstant}
If $G \in \mathcal{A}$, then $\gamma_u(G) \leqs 15$.
\end{prop}

\begin{proof}
First assume that $G = {\rm P\O}^+_{2r}(q)$ where $r \geqs 5$ is odd. For $G = \O^+_{10}(2)$ we choose $s = 2^{-} \perp 8^{-}$ as in \cite[Proposition 6.3]{GK} and a straightforward computation shows that $\gamma_u(G) \leqs 5$ (see Section~\ref{sss:comp_prob}). Now assume $G \neq \O^+_{10}(2)$. Following \cite[Table II]{GK}, fix an element $s = (r-1)^- \perp (r+1)^-$ in $G$. From the proof of \cite[Proposition 4.1]{GK}, it follows that $\M(G,s) = \{H\}$, where $H$ is a reducible subgroup of type $O^-_{r-1}(q) \times O^-_{r+1}(q)$. By Proposition~\ref{p:ben}, we have
\[ 
b(G, G/H) \leqs \left\lfloor \frac{2r}{r-1} \right\rfloor +11 \leqs 13
\]
and thus Corollary~\ref{c:CriterionBase} implies that $\gamma_u(G) \leqs b(G, G/H) \leqs 13$.  

The other groups $G \in \mathcal{A}$ are handled in a very similar fashion. In each case we choose $s \in G$ as in \cite[Table II]{GK}, noting that $\mathcal{M}(G,s)=\{H\}$ for some reducible subgroup $H$ (as before, this follows from the proof of \cite[Proposition 4.1]{GK}). Once again, the desired bound follows from Proposition~\ref{p:ben}, and it is worth noting that there are no special cases that require direct computation.
\end{proof}

Next we turn to the bounds in part (iii) of Theorem \ref{t:ClassicalUDN}. First we establish the  lower bound.

\begin{prop}\label{p:ClassicalLower}
If $G \in \mathcal{B}$, then $\gamma_u(G) \geqs r$.
\end{prop}

\begin{proof}
Suppose $G=\O_{2r+1}(q)$ and $1 \ne g \in G$. Then $g$ fixes a non-zero vector $v \in V$, where $V$ is the natural module. Therefore, $g$ is contained in the subspace stabiliser $H = G_{\la v \ra}$, which is a maximal subgroup of $G$. There are two possibilities: either $v$ is a singular vector, in which case $H$ is a $P_1$ parabolic subgroup of $G$, or $v$ is non-singular and $H$ is a subgroup of type $O^{\pm}_{2r}(q)$. In view of \eqref{e:lb}, we get 
\[ 
b(G, G/H) \geqs \left\lceil \frac{\log|G|}{\log|G/H|} \right\rceil \geqs  r 
\] 
in both cases. It follows that every non-identity element of $G$ is contained in a maximal subgroup $H$ with $b(G,G/H) \geqs r$, so Corollary~\ref{c:LowerBound} implies that $\gamma_u(G) \geqs r$. 

A very similar argument applies when $G = {\rm Sp}_{2r}(q)$ with $q$ even, using the fact that every element of $G$ is contained in an orthogonal subgroup $O_{2r}^{\pm}(q)$. We omit the details.
\end{proof}

\begin{prop}\label{p:B00}
If $G = \O_{2r+1}(q) \in \mathcal{B}$, then $\gamma_u(G) \leqs 2r+12 \leqs 7r$.
\end{prop}

\begin{proof}
As in \cite[Table II]{GK}, fix a semisimple element $s = 1 \perp (2r)^{-}$ and note that 
$\mathcal{M}(G,s) = \{H\}$, where $H$ is the stabiliser of a non-singular $1$-space. By combining Corollary \ref{c:CriterionBase} and Proposition \ref{p:ben}, we deduce that
\[
r \leqs \gamma_u(G) \leqs b(G,G/H) \leqs 2r+12
\]
as required.
\end{proof}

To complete the proof of the bounds in parts (i), (ii) and (iii) in Theorem \ref{t:ClassicalUDN}, it remains to handle the groups $G = {\rm Sp}_{2r}(q) \in \mathcal{B}$. 

\subsubsection{Symplectic groups in even characteristic}

In this section we complete the proof of Theorem~\ref{t:ClassicalUDN} by showing that $\gamma_u(G) \leqs 7r$ for all $G = {\rm Sp}_{2r}(q) \in \mathcal{B}$. To do this, we require some preliminary lemmas and additional notation. Let $\bar{G} = {\rm Sp}_{2r}(K)$ be the ambient simple algebraic group over the algebraic closure $K$ of $\mathbb{F}_q$ and let $\bar{V}$ be the natural module for $\bar{G}$. For an element $x \in G$, we define $\nu(x)$ to be the codimension of the largest eigenspace of $x$ on $\bar{V}$.

To establish the desired bound $\gamma_u(G) \leqs 7r$, we will work with an element $g \in G$ of order $q^r+1$. This allows us to appeal to earlier work of Bereczky \cite{Ber} to determine the maximal overgroups of $g$ (see Lemma \ref{l:bdd}(i) for $r \geqs 5$) and we then establish upper bounds on the relevant fixed point ratios (see Lemmas \ref{l:obound} and \ref{l:bdd}(ii)). Finally, we use Lemma \ref{l:ProbMethod} to establish the required bound on $\gamma_u(G)$; the groups with $r \geqs 5$ are handled in Proposition \ref{p:bdd1}, with the remaining cases treated in Proposition \ref{p:bdd2}.

\begin{lem}\label{l:count}
Suppose $G = {\rm Sp}_{2r}(q) \in \mathcal{B}$. For $s \in \{1,\ldots, 2r-1\}$, let $N_s$ be the number of elements $x \in G$ of prime order with $\nu(x)=s$. Then
\[
N_s < q^{\frac{1}{2}(4rs-s^2+3s+5)}.
\]
\end{lem}

\begin{proof}
Fix a prime $t$ and let $N_{s,t}$ be the number of elements $x \in G$ of order $t$ with $\nu(x)=s$. By combining \cite[Corollary 3.38 and Proposition 3.40]{Bur2}, we deduce that
\[
N_{s,t} < q^{\frac{1}{2}(s+1)} \cdot 2\left(\frac{q}{q-1}\right)^{\frac{s}{2}}q^{\frac{1}{2}(4rs-s^2+1)} \leqs q^{\frac{1}{2}(4rs-s^2+2s+4)}.
\]
To complete the argument, we need to show that there are at most $q^{(s+1)/2}$ possibilities for $t$. Suppose $t$ is odd and fix an element $x \in G$ of order $t$ with $\nu(x)=s$. Let $i \geqs 1$ be minimal such that $t$ divides $q^i-1$. Set $c=i$ if $i$ is even, otherwise $c=2i$. Note that $t \leqs q^{c/2}+1$, so it suffices to show that $c \leqs s+1$.

If $s<r$ then $s$ is even and the $1$-eigenspace of $x$ has dimension $2r-s$. Therefore, ${\rm Sp}_{s}(q)$ contains an element of order $t$, so $c \leqs s$ and the result follows. 
Now assume $s \geqs r$. Once again, if $G$ contains an element of order $t$ whose $1$-eigenspace is $(2r-s)$-dimensional, then $c \leqs s$ and we are done. If not, then $x$ must have a non-trivial eigenvalue (in $\mathbb{F}_{q^i}$) with multiplicity $2r-s$. Therefore, $(2r-s)c \leqs 2r$ and thus $c \leqs s+1$ as required.
\end{proof}

\begin{lem}\label{l:obound}
Suppose $G = {\rm Sp}_{2r}(q) \in \mathcal{B}$. Let $x \in G$ be an element of prime order with $\nu(x)=s$ and let $H<G$ be a maximal subgroup of type $O_{2r}^{\e}(q)$. Then 
\[
{\rm fpr}(x,G/H) \leqs \frac{1}{q^{s}} + \frac{1}{q^r-1}.
\]
\end{lem}

\begin{proof}
Let $x \in H$ be an element of prime order $t$ with $\nu(x) = s$. For now, let us assume $t$ is odd, so $s \geqs 2$. Since two semisimple elements in $H$ are $H$-conjugate if and only if they have the same eigenvalues on $\bar{V}$, it follows that $x^G \cap H = x^H$ and thus
\[
{\rm fpr}(x,G/H) = \frac{|x^G \cap H|}{|x^G|} = \frac{|H|}{|G|}\frac{|C_G(x)|}{|C_H(x)|}.
\]
The centraliser orders $|C_G(x)|$ and $|C_H(x)|$ can be read off from \cite[Table 3.6]{Bur2} and we deduce that 
\[
{\rm fpr}(x,G/H) = \frac{|H|}{|G|}\frac{|{\rm Sp}_{e}(q)|}{|O^{\e'}_e(q)|},
\]
where $e \geqs 0$ is the dimension of the $1$-eigenspace of $x$ on $\bar{V}$ (if $e=0$ then we define ${\rm Sp}_{e}(q) = O^{\e'}_e(q) = 1$) and $\e'$ is a suitable choice of sign. Since $e \leqs 2r-s$ we deduce that
\[
{\rm fpr}(x,G/H) \leqs \frac{|H|}{|G|} \cdot \frac{1}{2}q^{r-s/2}(q^{r-s/2}+1) \leqs \frac{q^{r-s/2}(q^{r-s/2}+1)}{q^r(q^r-1)}
\]
and the desired bound quickly follows. 

Now assume $t=2$. Here we use the Aschbacher--Seitz \cite{AS} notation for involution class representatives, so $x = a_s$, $b_s$ or $c_s$. It is easy to see that the $G$-class and $H$-class of $x$ have the same label and thus $x^G \cap H = x^H$. The conjugacy class sizes $|x^H|$ and $|x^G|$ can be read off from the proof of \cite[Proposition 3.22]{Bur2} and the desired bound is easily established. For example, suppose $x=b_s$, in which case $s$ is odd. Here
\begin{align*}
|x^H| & = \frac{|O_{2r}^{\e}(q)|}{2|{\rm Sp}_{s-1}(q)||{\rm Sp}_{2r-2s}(q)|q^{2r(s-1)-3s^2/2+3s/2}} \\
|x^G| & = \frac{|{\rm Sp}_{2r}(q)|}{|{\rm Sp}_{s-1}(q)||{\rm Sp}_{2r-2s}(q)|q^{2rs-3s^2/2+s/2}}
\end{align*}
and thus
\[
{\rm fpr}(x,G/H) = \frac{1}{q^s}\left(1+\frac{\e}{q^r-\e}\right) \leqs \frac{1}{q^s}\left(1+\frac{1}{q^r-1}\right).
\]
The result follows.
\end{proof}

\begin{lem}\label{l:bdd}
Suppose $G = {\rm Sp}_{2r}(q) \in \mathcal{B}$ with $r \geqs 5$, and let $g \in G$ be an element of order $q^r+1$. 
\begin{itemize}\addtolength{\itemsep}{0.2\baselineskip}
\item[{\rm (i)}] $\M(G,g) = \{H, H_1, \ldots, H_{\ell}\}$, where $H = O_{2r}^{-}(q)$ and $H_i = {\rm Sp}_{2r/k}(q^k).k$ for some prime divisor $k$ of $r$ (one subgroup for each prime). 
\item[{\rm (ii)}] If $x \in G$ has prime order and $\nu(x) = s \geqs 3$, then 
\begin{equation}\label{eq:qq0}
\sum_{i=1}^{\ell} {\rm fpr}(x,G/H_i) < \frac{1}{q^{s}} + \frac{1}{q^r-1}.
\end{equation}
\end{itemize}
\end{lem}

\begin{proof}
The description of the maximal overgroups in part (i) follows from the proof of \cite[Proposition 5.8]{BGK} (also see \cite{Ber}). Now consider (ii). By \cite[Corollary 3.38]{Bur2} we have $|x^G|>\a$, where
\begin{equation}\label{e:al}
\a = \frac{1}{2}\left(\frac{q}{q+1}\right)q^{\b},\;\; \b = \left\{\begin{array}{ll} s(2r-s) & s< r \\ rs & s \geqs r \end{array}\right.
\end{equation}
(since $q$ is even, we can replace the coefficient $\frac{1}{4}$ in \cite[Corollary 3.38]{Bur2} by 
$\frac{1}{2}$). Moreover, by the main theorem of \cite{Bur} we have
\[
{\rm fpr}(x,G/H_i) < |x^G|^{-\frac{1}{2}+\frac{1}{2r}+\frac{1}{2r+2}}
\]
for all $i$. Since $\ell \leqs \log_2r$, we deduce that
\[
\sum_{i=1}^{\ell} {\rm fpr}(x,G/H_i) < \log_2r \cdot \a^{-\frac{1}{2}+\frac{1}{2r}+\frac{1}{2r+2}}.
\]
In view of the conditions $r \geqs 5$ and $s \geqs 3$, one checks that this upper bound is less than $q^{-s}+(q^r-1)^{-1}$ and the result follows.
\end{proof}

\begin{prop}\label{p:bdd1}
If $G = {\rm Sp}_{2r}(q)\in \mathcal{B}$ and $r \geqs 5$, then $\gamma_u(G) \leqs 7r$.
\end{prop}

\begin{proof}
As in Lemma \ref{l:bdd}, let $g \in G$ be an element of order $q^r+1$ and define $\widehat{Q}(G,g,7r)$ as in \eqref{e:qsc}. In view of Lemma \ref{l:ProbMethod}, it suffices to show that $\widehat{Q}(G,g,7r)<1$. To do this, it will be convenient to write
\[
\widehat{Q}(G,g,7r) = \widehat{Q}_1 + \widehat{Q}_2 + \widehat{Q}_3,
\]
where $\widehat{Q}_1$ and $\widehat{Q}_2$ are the contributions to $\widehat{Q}(G,g,7r)$ from the elements $x \in G$ of prime order with $\nu(x)=1$ and $2$, respectively, and $\widehat{Q}_3$ is the contribution from the remaining elements of prime order in $G$. We will estimate each $\widehat{Q}_i$ in turn. By Lemma \ref{l:bdd}(i), $\M(G,g) = \{H, H_1, \ldots, H_{\ell}\}$, where $H = O_{2r}^{-}(q)$ and $H_i = {\rm Sp}_{2r/k}(q^k).k$ for some prime divisor $k$ of $r$ (one subgroup for each prime). As before, we use the notation of \cite{AS} for involution class representatives.

First consider $\widehat{Q}_1$. There is a unique class of elements $x \in G$ of prime order with $\nu(x)=1$, namely the involutions of type $b_1$ (that is, the transvections in $G$). Here $|x^G|=q^{2r}-1$ and it is very easy to check that 
\[
{\rm fpr}(x,G/H) = \frac{1}{q} + \frac{1}{q(q^r-1)}
\]
and ${\rm fpr}(x,G/H_i) = 0$ for all $i$, whence
\begin{equation}\label{e:q1}
\widehat{Q}_1 = (q^{2r}-1)\cdot \left(\frac{1}{q} + \frac{1}{q(q^r-1)}\right)^{7r}.
\end{equation}

Now let us turn to $\widehat{Q}_2$. If $x$ is an $a_2$-involution, then 
\[
|x^G| = \frac{(q^{2r-2}-1)(q^{2r}-1)}{(q^2-1)} = u_1, \;\; {\rm fpr}(x,G/H) = \frac{q^{r-2}-1}{q^r-1} = v_1
\] 
and ${\rm fpr}(x,G/H_i) = 0$ for all $i$. Similarly, if $x=c_2$ then 
\[
|x^G|=(q^{2r-2}-1)(q^{2r}-1)= u_2,\;\;
{\rm fpr}(x,G/H) = \frac{1}{q^2} + \frac{1}{q^2(q^r-1)} = v_2
\]
and ${\rm fpr}(x,G/H_i) = 0$, unless $r$ is even and $H_i = {\rm Sp}_{r}(q^2).2$, in which case 
\[
{\rm fpr}(x,G/H_i) = \frac{q^{2r}-1}{|x^G|} = \frac{1}{q^{2r-2}-1} = w_2.
\]

Now assume $x \in G$ has odd prime order $t$ and $\nu(x)=2$, so $t$ divides $q^2-1$ and $G$ has exactly $(t-1)/2$ distinct conjugacy classes of such elements. 
In particular, if $t$ divides $q-\e$, then $G$ contains at most 
\[
\frac{1}{2}(q-\e) \cdot \frac{|{\rm Sp}_{2r}(q)|}{|{\rm Sp}_{2r-2}(q)|{|\rm GL}_{1}^{\e}(q)|} = \frac{1}{2}q^{2r-1}(q^{2r}-1)
\]
elements of order $t$. Since $q-\e$ has fewer than $\log_2(q-\e)$ odd prime divisors, it follows that $G$ contains at most 
\[
\log_2(q^2-1) \cdot \frac{1}{2}q^{2r-1}(q^{2r}-1)
\]
such elements. Now ${\rm fpr}(x,G/H_i) = 0$ for all $i$ and 
\[
{\rm fpr}(x,G/H) \leqs \frac{|O_{2r}^{-}(q)|}{|O_{2r-2}^{+}(q)||{\rm GU}_{1}(q)|} \cdot \frac{|{\rm Sp}_{2r-2}(q)||{\rm GU}_{1}(q)|}{|{\rm Sp}_{2r}(q)|} = \frac{q^{r-1}+1}{q(q^r-1)}.
\]
Putting all this together, we conclude that
\begin{equation}\label{e:q2}
\widehat{Q}_2 < u_1v_1^{7r} + u_2(v_2+w_2)^{7r} + \log_2(q^2-1) \cdot \frac{1}{2}q^{2r-1}(q^{2r}-1) \cdot \left(\frac{q^{r-1}+1}{q(q^r-1)}\right)^{7r}.
\end{equation}

Finally, let us consider $\widehat{Q}_3$. We will use the inequality
\[
(a+b)^n \leqs 2^{n-1}(a^n+b^n),
\]
which is valid for all positive real numbers $a,b,n$ with $n \geqs 1$. By combining Lemmas \ref{l:count}, \ref{l:obound} and \ref{l:bdd}, we get
\begin{align*}
\widehat{Q}_3 & < \sum_{s=3}^{2r-1} q^{\frac{1}{2}(4rs-s^2+3s+5)} \cdot \left(\frac{2}{q^s}+\frac{2}{q^r-1}\right)^{7r} \\
              & < \sum_{s=3}^{2r-1} q^{\frac{1}{2}(4rs-s^2+3s+5)} \cdot q^{14r-1}\left(q^{-7rs}+q^{-7r(r-1)}\right) \\
              & = q^{14r+\frac{3}{2}}\left(\sum_{s=3}^{2r-1} q^{\frac{1}{2}(3s-s^2-10rs)}   \right) + q^{21r-7r^2+\frac{3}{2}}\left(\sum_{s=3}^{2r-1}  q^{\frac{1}{2}(4rs-s^2+3s)} \right) \\
              & < q^{14r+\frac{3}{2}}\cdot q^{-15r+1} + q^{21r-7r^2+\frac{3}{2}} \cdot q^{2r^2+3r-1}
\end{align*}
and thus
\begin{equation}\label{e:q3}
\widehat{Q}_3 < q^{-r+\frac{5}{2}} + q^{24r-5r^2+\frac{1}{2}}. 
\end{equation}

By combining the expression for $\widehat{Q}_1$ in \eqref{e:q1} with the bounds on $\widehat{Q}_2$ and $\widehat{Q}_3$ in \eqref{e:q2} and \eqref{e:q3}, it is easy to check that 
$\widehat{Q}(G,g,7r) <1$ for all $r \geqs 5$. This completes the proof of the proposition.
\end{proof}

Finally, we show that the desired bound also holds when $r \in \{2,3,4\}$.

\begin{prop}\label{p:bdd2}
If $G = {\rm Sp}_{2r}(q) \in \mathcal{B}$ and $r \in \{2,3,4\}$, then $\gamma_u(G) \leqs 7r$.
\end{prop}

\begin{proof}
As before, fix an element $g \in G$ of order $q^r+1$ and note that the description of $\mathcal{M}(G,g)$ in Lemma \ref{l:bdd}(i) still holds. In addition, if we define $\widehat{Q}_i$ as above then the expression for $\widehat{Q}_1$ in \eqref{e:q1} is still valid. Similarly, if $r \in \{3,4\}$ then we get the upper bound on $\widehat{Q}_2$ in \eqref{e:q2} (and we can set $w_2=0$ when $r=3$).

First assume $r=4$ and $q \geqs 4$, so $\M(G,g) = \{O_{8}^{-}(q), {\rm Sp}_{4}(q^2).2\}$ and $\ell=1$ in the notation of Lemma \ref{l:bdd}. Moreover, one can check that the upper bound in \eqref{eq:qq0} holds (the same proof goes through unchanged) and thus
\[
\widehat{Q}_3 < |G|\cdot \left(\frac{2}{q^3}+\frac{2}{q^4-1}\right)^{28}< q^{36}\left(\frac{2}{q^3}+\frac{2}{q^4-1}\right)^{28}.
\]
It is now easy to check that $\widehat{Q}(G,g,28)<1$. The case $(r,q)=(4,2)$ can be handled using \textsf{GAP} (see Section \ref{sss:comp_prob}) and we get $\gamma_u(G) \leqs 10$.  

The case $r=3$ is very similar. Here $\M(G,g) = \{O_{6}^{-}(q), {\rm Sp}_{2}(q^3).3\}$ and the bound in \eqref{eq:qq0} still holds. Indeed, the main theorem of \cite{Bur} implies that 
\[
\sum_{i=1}^{\ell}{\rm fpr}(x,G/H_i)  = {\rm fpr}(x,G/H_1) < \a^{-\frac{1}{2}+\frac{1}{6}} = \a^{-\frac{1}{3}},
\]
where $H_1 = {\rm Sp}_{2}(q^3).3$ and $\a$ is defined as in \eqref{e:al}, and one checks that this is less than $q^{-s}+(q^3-1)^{-1}$. Therefore,
\[
\widehat{Q}_3 < |G|\cdot \left(\frac{2}{q^3}+\frac{2}{q^3-1}\right)^{21}< q^{21}\left(\frac{2}{q^3}+\frac{2}{q^3-1}\right)^{21}
\]
and the result follows.

Finally, suppose $r=2$ and $q \geqs 4$. Write $\M(G,g) = \{H,H_1\}$, where $H = O_4^{-}(q)$ and $H_1 = {\rm Sp}_{2}(q^2).2$. As noted above, the expression for $\widehat{Q}_1$ in \eqref{e:q1} is still valid. Now consider $\widehat{Q}_2$. If $x=a_2$ then ${\rm fpr}(x,G/H)=0$ and ${\rm fpr}(x,G/H_1) = q/(q^2-1)$ (the involutory field automorphisms of ${\rm Sp}_{2}(q^2)$ are $a_2$-involutions). Similarly, if $x=c_2$ then $|x^G|=u_2$, ${\rm fpr}(x,G/H)=v_2$ and ${\rm fpr}(x,G/H_1)=w_2$ as before, so the upper bound in \eqref{e:q2} holds, with $v_1 = q/(q^2-1)$. 

To complete the proof of the proposition, we may assume $x \in G$ has prime order $t$ and $\nu(x)=3$. Here $t$ is odd and $x$ is regular. Let $i \in \{1,2,4\}$ be minimal such that $t$ divides $q^i-1$. We consider each possibility for $i$ in turn.

Suppose $i=4$, so $t$ divides $q^2+1$ and we see that there are at most $\log_2(q^2+1)$ possibilities for $t$. In addition, for a fixed prime $t$, there are at most $\frac{1}{4}(t-1) \leqs \frac{1}{4}q^2$ distinct $G$-classes of elements of order $t$, each of which has size  
\[
\frac{|{\rm Sp}_{4}(q)|}{|{\rm GU}_{1}(q^2)|} = q^4(q^2-1)^2.
\]
It is straightforward to check that 
\[
{\rm fpr}(x,G/H) = {\rm fpr}(x,G/H_1) = \frac{2}{q^2(q^2-1)}.
\]
Finally, suppose $i \in \{1,2\}$. Here neither $H$ nor $H_1$ contains any regular semisimple elements of order $t$, so ${\rm fpr}(x,G/H) = {\rm fpr}(x,G/H_1) = 0$. Putting this together, we conclude that 
\[
\widehat{Q}_3 < \log_2(q^2+1) \cdot \frac{1}{4}q^2 \cdot q^4(q^2-1)^2 \cdot \left(\frac{2}{q^2(q^2-1)}\right)^{14}
\]
and the result follows. 
\end{proof}

\subsubsection{General case}
To complete the proof of Theorem \ref{t:ClassicalUDN}, we need to verify the bound
\[
\gamma_u(G) \leqs  7r+56.
\]
In view of our earlier work, we may assume that $r>1$ and $G \not\in \mathcal{A} \cup \mathcal{B}$. It will be convenient to handle some small groups separately, and with this in mind we define 
\[
\mathcal{C} = \{ {\rm L}_9(2), {\rm L}_8(2), {\rm P}\O^+_8(3),  \O^+_8(2),  {\rm L}_{7}(2),  {\rm PSp}_6(3),  {\rm U}_6(2),  {\rm L}_4(2),  {\rm U}_4(3),  {\rm U}_4(2),   {\rm U}_3(5) \}.
\]

\begin{prop}\label{p:bdc}
If $G \in \mathcal{C}$, then $\gamma_u(G) \leqs c$, where $c$ is as follows:

{\small
\[
{\renewcommand{\arraystretch}{1.1}
\begin{array}{lccccccccccc} 
\hline
G        & {\rm L}_9(2) & {\rm L}_8(2)  & {\rm P}\O^+_8(3) & \O^+_8(2)  & {\rm L}_{7}(2) & {\rm PSp}_6(3) & {\rm U}_6(2) & {\rm L}_4(2) & {\rm U}_4(3)   & {\rm U}_4(2) & {\rm U}_3(5)  \\
c                  & 7            & 9    & 19               & 26    & 6     & 3              & 6            & 4            & 16             & 8            & 13         \\
s      &  4 \oplus 5   & 3 \oplus 5 &  {\tt 14A}        & {\tt 15A}  & {\tt 105A} & {\tt 14A}      & {\tt 11A}    & {\tt 15A}    & {\tt 9A}       & {\tt 9A}     & {\tt 13A}  \\
\hline
\end{array}}
\]}

\vspace{1mm}

\noindent In particular, $\gamma_u(G) \leqs  7r+56$.
\end{prop}

\begin{proof}
We implement the probabilistic method computationally (see Section~\ref{sss:comp_prob}), working with an element $s \in G$ in the conjugacy class specified in the table. We use 
\textsc{Magma} to handle the groups ${\rm L}_9(2)$ and ${\rm L}_8(2)$, and \textsf{GAP} for the remaining cases. In this way, for $G \ne \O^+_8(2)$, one can check that $\widehat{Q}(G,s,c)<1$ for the stated value of $c$, whence $\gamma_u(G) \leqs c$ as required. 

The case $G = \O^+_8(2)$ requires more attention. If $s$ is in {\tt 15A}, then 
\begin{equation}\label{e:fxs}
F(x,s) \coloneqq \sum_{H \in \M(G,s)} \fpr(x,G/H) > 1
\end{equation}
for some elements $x \in G$ of prime order and thus $\widehat{Q}(G,s,d)>1$ for all $d \geqs 1$. However, the proof of \cite[Proposition~6.2]{GK} gives
\[
P(x,s) =  \frac{|\{z \in s^G \,:\, G \ne \< x,z \>\}|}{|s^G|} < \frac{7}{10}
\]
for all elements $x \in G$ of prime order. Therefore, $P(x,s) \leqs \min\{F(x,s), 7/10\} =: Q(x,s)$ and thus 
\[
Q(G,s,26) \leqs \sum_{i=1}^{k}|x_i^G| \cdot P(x_i,s)^{26} \leqs \sum_{i=1}^{k}|x_i^G| \cdot Q(x_i,s)^{26}< 1
\]
(see \eqref{e:pxs}). This gives $\gamma_u(G) \leqs 26$ as claimed.
\end{proof}

For the remainder, we can assume $r>1$ and $G \not\in \mathcal{A} \cup \mathcal{B} \cup\mathcal{C}$. Define an integer $R(G)$ as follows:
\[
{\renewcommand{\arraystretch}{1.1}
\begin{array}{lccccc}
\hline
G      & {\rm L}_{r+1}(q) & {\rm U}_{r+1}(q) & {\rm PSp}_{2r}(q) & {\rm P\O}^+_{2r}(q) & {\rm P\O}^-_{2r}(q) \\ 
R(G)   & 6            & 7            & 4              & 4                & 6                \\
\hline
\end{array}}
\]

\begin{rem}\label{r:obs}
Suppose $G \not\in \mathcal{A} \cup \mathcal{B} \cup \mathcal{C}$, $r < R(G)$ and $G \ne {\rm Sp}_{4}(2)'\cong A_6$. Then $|\M(G,s)|=1$ for the element $s \in G$ identified in the proof of Theorem~\ref{t:ClassicalMax}. 
\end{rem}

The following lemma on fixed point ratios is our key tool in the proof of Theorem \ref{t:ClassicalUDN}. The proof uses several results on fixed point ratios for primitive actions of finite simple classical groups. For example, if $H$ acts reducibly on the natural module $V$ then we appeal to the bounds on ${\rm fpr}(x,G/H)$ obtained by Guralnick and Kantor in \cite[Section 3]{GK}. For an irreducible subgroup $H$, we apply the main theorem of \cite{Bur}. For instance, if $H$ is irreducible and the rank $r$ of $G$ is large enough, then \cite[Corollary 2]{Bur} states that
\begin{equation}\label{eq:fbd}
{\rm fpr}(x,G/H) < 2q^{-\frac{r(r-1)}{r+1}}
\end{equation}
for all $x \in G$ of prime order (in particular, this holds if $\dim V\geqs 7$).  In the statement of the lemma, we define $F(x,s)$ as in \eqref{e:fxs}.

\begin{lem}\label{l:ClassicalFPRs}
Let $G \not\in \mathcal{A} \cup \mathcal{B} \cup \mathcal{C}$ be a finite simple classical group of rank $r \geqs R(G)$. Then there exists $s \in G$ such that for all elements $x \in G$ of prime order
\[
F(x,s) < \min\left\{\dfrac{c}{q^2}, \dfrac{c}{q^{r/2-3/2}}\right\}
\]
where $c=3$ unless $G = {\rm L}_{r+1}(q)$, in which case $c=4$.
\end{lem}

\begin{proof}
Throughout, let $x \in G$ be an element of prime order. We partition the proof into four cases:
\begin{itemize}\addtolength{\itemsep}{0.2\baselineskip}
\item[{\rm (a)}] $G \in \{ {\rm L}_{r+1}(q) \,:\, r \geqs 6 \} \cup \{ {\rm U}_{r+1}(q) \,:\, \text{$r \geqs 8$ even} \}$.
\item[{\rm (b)}] $G \in \{ {\rm P}\O^-_{2r}(q) \,:\, r \geqs 6 \} \cup \{ {\rm PSp}_{2r}(q) \,:\, \text{$r \geqs 6$ even, $q$ odd} \}$.
\item[{\rm (c)}] $G = {\rm P}\O^+_{2r}(q)$, where $r \geqs 6$ is even.
\item[{\rm (d)}] $G = {\rm PSp}_{8}(q)$ with $q$ odd, or $G = {\rm P\O}_{8}^{+}(q)$ and $q \geqs 4$.
\end{itemize}

First consider (a). The case $G = {\rm L}_{11}(2)$ requires special attention. If $s\in G$ has order $2^{11}-1$ then the proof of \cite[Proposition~6.3]{GK} gives $\M(G,s) = \{H\}$ with $H$ a field extension subgroup of type ${\rm GL}_{1}(2^{11})$. By applying the bound in \eqref{eq:fbd}, we deduce that $F(x,s) = \fpr(x,G/H) < 2^{-79/11}$ and the result follows. 

For the other groups in (a), let $s \in G$ be the element defined in \cite[Table~II]{GK}. As explained in the proof of \cite[Proposition~4.1]{GK}, the maximal overgroups of $s$ are the obvious reducible subgroups, and \cite[Propositions 3.15 and 3.16]{GK} supply upper bounds on the associated fixed point ratios. It is now straightforward to verify the desired bound. For example, suppose $G={\rm L}_{r+1}(q)$ with $r \geqs 6$. As in \cite[Table~II]{GK}, we take  
\[ 
s = 
\left\{
\begin{array}{ll}
\frac{r+2}{2} \oplus \frac{r}{2}   & \text{if $r$ is even} \\
\frac{r+5}{2} \oplus \frac{r-3}{2} & \text{if $r \equiv 1 \imod{4}$} \\
\frac{r+3}{2} \oplus \frac{r-1}{2} & \text{if $r \equiv 3 \imod{4}$}, \\
\end{array}
\right.
\]  
noting that $\mathcal{M}(G,s) = \{H,K\}$ where $H$ and $K$ are the stabilisers of appropriate $k$- and $(r+1-k)$-spaces with $k \leqs r/2$. By \cite[Proposition~3.1(i)]{GK}, if $L \leqs G$ is the stabiliser of an $\ell$-space with $\ell \leqs (r+1)/2$, then ${\rm fpr}(y,G/H) < 2q^{-\ell}$ for all $1 \ne y \in G$. Therefore, 
\[
F(x,s) = \fpr(x,G/H) + \fpr(x^{\tau},G/H) < \min\left\{\frac{4}{q^2}, \frac{4}{q^{r/2-3/2}} \right\},
\]
where $\tau \in {\rm Aut}(G)$ is an involutory graph automorphism. A similar argument applies when $G = {\rm U}_{r+1}(q)$ with $r \geqs 8$ even and we omit the details.

Next consider the groups in (b). Let $s \in G$ be a Singer cycle (that is, $s$ generates an irreducible cyclic subgroup of maximal possible order). By the main theorem of Bereczky \cite{Ber}, the maximal overgroups of $s$ are field extension subgroups and \cite{Bur} provides  upper bounds on the associated fixed point ratios. We will assume $G = {\rm P}\O^-_{2r}(q)$ with $r \geqs 6$; the other case is very similar.

By \cite{Ber}, the members of $\M(G,s)$ are subgroups of type ${\rm GU}_{r}(q)$ and $O^-_{2r/k}(q^k)$, where $k$ is a prime divisor of $r$. In addition, if $qr$ is odd, then there are also field extension subgroups of type $O_{r}(q^2)$.  A straightforward calculation shows that $\M(G,s)$ contains exactly one of each such subgroup. From \cite[Corollary 3.38]{Bur2} we get 
\[
|x^G|>\frac{1}{4}\left(\frac{q}{q+1}\right)q^{4r-6}
\]
and thus \cite[Theorem 1]{Bur} implies that
\begin{equation}\label{eq:nk}
{\rm fpr}(x,G/H) < |x^G|^{-\frac{1}{2}+\frac{1}{2r}+\frac{1}{2r-2}} < \left(\frac{1}{4}\left(\frac{q}{q+1}\right)q^{4r-6}\right)^{-\frac{1}{2}+\frac{1}{2r}+\frac{1}{2r-2}} < \frac{2}{q^{r-1}}
\end{equation}
for each $H \in \mathcal{M}(G,s)$. Therefore,
\[
F(x,s) < \frac{2(2+\log_2r)}{q^{r-1}} < \min\left\{\frac{3}{q^2}, \frac{3}{q^{r/2-3/2}} \right\}.
\]

Next let us turn to case (c), so $G = {\rm P}\O^+_{2r}(q)$ and $r \geqs 6$ is even. Fix an element 
$s = (r-2)^- \perp (r+2)^- \in G$. Then \cite[Proposition~5.14]{BGK} implies that 
$\M(G,s) = \{ L, H_1, H_2 \}$, where $L$ is a reducible subgroup of type $O_{r-2}^-(q) \times O_{r+2}^-(q)$ and $H_1, H_2$ are field extension subgroups of type $O^+_r(q^2)$. By applying 
\cite[Proposition~3.16]{GK} and the bound in \eqref{eq:fbd}, we get  
\begin{equation}\label{eq:nk2}
F(x,s) < \left(\frac{3}{q^{r-2}} + \frac{1}{q^{r-1}} + \frac{1}{q^{r/2}}\right) + 4q^{-\frac{r(r-1)}{r+1}} < \min\left\{\frac{3}{q^2}, \frac{3}{q^{r/2-3/2}} \right\}.
\end{equation}

Finally, we handle the two cases in (d). First assume $G = {\rm PSp}_{8}(q)$ and $q$ is odd.
Take $s = 2 \perp 6 \in G$ and note that the proof of \cite[Proposition 4.1]{GK} implies that 
$\M(G,s) = \{ H \}$, where $H$ is of type ${\rm Sp}_{2}(q) \times {\rm Sp}_{6}(q)$. By applying \cite[Proposition~3.5]{Harper} we deduce that $F(x,s) < 2q^{-2}$ and the result follows. 
Finally, suppose $G = {\rm P}\O^+_8(q)$ with $q \geqs 4$. Let $s \in G$ be an element of order $(q^2+1)/(q-1,2)$. By \cite[Proposition~5.15]{BGK}, $\M(G,s)$ contains three subgroups of type 
$O_4^-(q) \wr S_2$, plus an additional subfield subgroup of type $O_8^-(q^{1/2})$ if $q$ is a square. Now $|x^G| \geqs (q^2+1)^2(q^6-1)$ (minimal if $q$ is even and $x$ is an $a_2$ involution), so the main theorem of \cite{Bur} implies that
\[
F(x,s) < 4\left((q^2+1)^2(q^6-1)\right)^{-\frac{1}{2}+\frac{1}{8}} < \frac{3}{q^2}
\]
and the result follows. This completes the proof of the lemma.
\end{proof}

We are now ready to complete the proof of Theorem \ref{t:ClassicalUDN}.

\begin{prop}\label{p:ClassicalLinearUpper}
Let $G$ be a finite simple classical group of rank $r$. Then
\[
\gamma_u(G) \leqs 7r+56.
\]
\end{prop}

\begin{proof}
We have already verified the bounds in parts (i), (ii) and (iii) of Theorem \ref{t:ClassicalUDN}, so we may assume that $r>1$ and $G \not\in \mathcal{A} \cup \mathcal{B}$. In addition, we can assume that $G \not\in \mathcal{C}$ (see Proposition \ref{p:bdc}) and $G \ne {\rm Sp}_{4}(2)'\cong A_6$ (see the remark following the statement of Theorem \ref{t:ClassicalUDN}).  Note that if $G$ is linear or unitary, then $|G| < q^{r^2+2r}$. In general, $|G|< q^{2r^2+\e r}$, where $\e=1$ if $q$ is odd, otherwise $\e=-1$ (since the symplectic groups in even characteristic are contained in $\mathcal{B}$). We will use the notation from Lemma \ref{l:ProbMethod}.

First assume $r \geqs 7$ and $(r,q) \ne (7,2), (8,2)$. Choose $s \in G$ as in Lemma~\ref{l:ClassicalFPRs}. If $G$ is linear or unitary, then we deduce that
\[
\what{Q}(G,s,7r+56) < q^{r^2+2r}\cdot\left(\frac{4}{q^{r/2-3/2}}\right)^{7r+56} < 1,
\]
so Lemma \ref{l:ProbMethod} implies that $\gamma_u(G) \leqs 7r+56$. Similarly, if $G$ is symplectic or orthogonal, then
\[
\what{Q}(G,s,7r+56) < q^{2r^2+\e r}\cdot\left(\frac{3}{q^{r/2-3/2}}\right)^{7r+56} < 1,
\]
and once again we conclude that $\gamma_u(G) \leqs 7r+56$.

Now assume $(r,q)=(7,2)$, so $\O_{14}^{-}(2)$ (recall that $G \not\in \mathcal{A} \cup \mathcal{B} \cup \mathcal{C}$). As in the proof of Lemma~\ref{l:ClassicalFPRs}, let $s \in G$ be a Singer cycle (so $|s| = 2^7+1$) and note that $\M(G,s) = \{H_1,H_2\}$, where $H_1$ is of type ${\rm GU}_{7}(2)$ and $H_2$ is of type $O_2^{-}(2^7)$. In view of the  upper bound in \eqref{eq:nk}, we deduce that $F(x,s) < 2^{-4}$ and this immediately implies that $\what{Q}(G,s,7r+56) < 1$. 
 
Similar arguments apply when $(r,q)=(8,2)$. If $G={\rm U}_{9}(2)$, then $|G|<2^{80}$ and the upper bound on $F(x,s)$ in Lemma \ref{l:ClassicalFPRs} is sufficient. If $G = \O_{16}^{-}(2)$ then we choose $s$ as in the proof of Lemma \ref{l:ClassicalFPRs}, so the subgroups in $\mathcal{M}(G,s)$ are of type ${\rm GU}_{8}(2)$ and $O_{8}^{-}(4)$ (one subgroup of each type). From the bound in \eqref{eq:nk} we get $F(x,s)< 2^{-5}$ and the desired result quickly follows. Similarly, if $G = \O_{16}^{+}(2)$ then the upper bound in \eqref{eq:nk2} gives $F(x,s)<2^{-2}$ and this implies that $\what{Q}(G,s,7r+56) < 1$ as required.

Finally, let us assume $r<7$. First suppose $r \geqs R(G)$, in which case we choose $s \in G$ as in Lemma \ref{l:ClassicalFPRs}. If $G = {\rm L}_{7}(q)$, then $r=6$ and
\[
\what{Q}(G,s,7r+56) < q^{r^2+2r}\cdot\left(\frac{4}{q^{2}}\right)^{7r+56},
\]
which is less than $1$ if $q \geqs 3$ (note that ${\rm L}_{7}(2)$ is in the collection $\mathcal{C}$). Similarly, if $G \ne {\rm L}_{7}(q)$ then 
\[
\what{Q}(G,s,7r+56) < q^{2r^2+\e r}\cdot\left(\frac{3}{q^{2}}\right)^{7r+56}
\]
and it just remains to handle the groups $\O_{12}^{\pm}(2)$ (for example, if $(r,q)=(5,2)$ and $r \geqs R(G)$ then $G = {\rm Sp}_{10}(2)$ or $\O_{10}^{+}(2)$, both of which belong to $\mathcal{A} \cup \mathcal{B}$). For $G = \O_{12}^{-}(2)$ we take a Singer cycle $s \in G$, in which case \eqref{eq:nk} implies that $F(x,s)<3/16$ and we get $\what{Q}(G,s,7r+56) < 1$. Similarly, if $G = \O_{12}^{+}(2)$ and $s = 4^{-}\perp 8^{-}$ then the upper bound on $F(x,s)$ in \eqref{eq:nk2} is good enough to give $\what{Q}(G,s,7r+56) < 1$.

To complete the proof, we may assume $r<R(G)$. As noted in Remark~\ref{r:obs}, there exists $s \in G$ such that $\M(G,s)=\{H\}$ for some subgroup $H$, so $\gamma_u(G) \leqs b(G,G/H)$ by Corollary \ref{c:CriterionBase}. If $H$ is reducible, then $b(G,G/H) \leqs \dim V + 11$ by Proposition \ref{p:ben}, where $V$ is the natural module for $G$. Otherwise, $b(G,G/H) \leqs 5$ by the main theorem of \cite{B07}. In all cases, the result follows.
\end{proof}

This completes the proof of Theorem~\ref{t:ClassicalUDN}, and hence Theorem~4.

\begin{rem}\label{r:final}
It is easy to improve the bound in Proposition \ref{p:ClassicalLinearUpper} in special cases. For example, we have already shown that better bounds hold when $G$ is one of the groups covered by parts (i), (ii) or (iii) of Theorem \ref{t:ClassicalUDN}. In other cases, the proof of Proposition \ref{p:ClassicalLinearUpper} also yields better bounds. For instance, suppose $G = {\rm U}_{r+1}(q)$, with $r \geqs 7$ and $q \geqs 4$. Then $|G|<q^{r^2+2r}$ and one checks that
\[
\what{Q}(G,s,2r+40) < q^{r^2+2r}\cdot\left(\frac{3}{q^{r/2-3/2}}\right)^{2r+40} < 1
\]
for the element $s \in G$ given in Lemma \ref{l:ClassicalFPRs}. Therefore, $\gamma_u(G) \leqs 2r+40$.
\end{rem}


\end{document}